\documentclass[a4paper]{article}
\usepackage{amssymb}
\usepackage{amsmath}
\usepackage{mathrsfs}
\usepackage{amsfonts}
\usepackage{color}
\usepackage{vmargin}
\usepackage{amsthm}

\long\def\symbolfootnote[#1]#2{\begingroup%
\def\thefootnote{\fnsymbol{footnote}}\footnote[#1]{#2}\endgroup}

\setmarginsrb{20mm}{20mm}{20mm}{20mm}{10mm}{10mm}{10mm}{10mm}

\def\Xint#1{\mathchoice
{\XXint\displaystyle\textstyle{#1}}%
{\XXint\textstyle\scriptstyle{#1}}%
{\XXint\scriptstyle\scriptscriptstyle{#1}}%
{\XXint\scriptscriptstyle\scriptscriptstyle{#1}}%
\!\int}
\def\XXint#1#2#3{{\setbox0=\hbox{$#1{#2#3}{\int}$ }
\vcenter{\hbox{$#2#3$ }}\kern-.6\wd0}}

\def\dashint{\Xint-}

\DeclareMathOperator{\dist}{dist}
\DeclareMathOperator{\diam}{diam}

\def \harm{\mathcal{H}}
\def \wharm{w\mathcal{H}}
\def \subh{\mathcal{S_{-}H}}
\def \superh{\mathcal{S^{+}H}}

\def \muballxr{\mu(B(x,r))}

\def \muball{\mu(B)}
\def\vint{\mathop{\mathchoice%
          {\setbox0\hbox{$\displaystyle\intop$}\kern 0.22\wd0%
           \vcenter{\hrule width 0.6\wd0}\kern -0.82\wd0}%
          {\setbox0\hbox{$\textstyle\intop$}\kern 0.2\wd0%
           \vcenter{\hrule width 0.6\wd0}\kern -0.8\wd0}%
          {\setbox0\hbox{$\scriptstyle\intop$}\kern 0.2\wd0%
           \vcenter{\hrule width 0.6\wd0}\kern -0.8\wd0}%
          {\setbox0\hbox{$\scriptscriptstyle\intop$}\kern 0.2\wd0%
           \vcenter{\hrule width 0.6\wd0}\kern -0.8\wd0}}%
          \mathopen{}\int}

\newcommand{\notimplies}{%
  \mathrel{{\ooalign{\hidewidth$\not\phantom{=}$\hidewidth\cr$\implies$}}}}

\newcommand{\N}{{\mathbb N}}
\newcommand{\R}{{\mathbb R}}

\newcommand{\ep}{\epsilon}
\newcommand{\Om}{\Omega}

\newcommand{\kom}[1]{}
\renewcommand{\kom}[1]{{\bf [#1]}}

\definecolor{blau}{rgb}{0.1,0.0,0.9}

\newcounter{komcounter}
\numberwithin{komcounter}{section}

\theoremstyle{plain}
\newtheorem{theorem}{Theorem}[section]
\newtheorem{prop}{Proposition}[section]
\newtheorem{lemma}{Lemma}[section]
\newtheorem{corol}{Corollary}[section]
\theoremstyle{definition}
\newtheorem{definition}{Definition}[section]
\newtheorem{example}{Example}
\newtheorem{remark}{\textnormal{\textbf{Remark}}}

\theoremstyle{definition}

\begin{document}

\title {Harmonic functions on metric measure spaces}

\author{
Tomasz Adamowicz{\small{$^1$}}
\\
\it\small Institute of Mathematics, Polish Academy of Sciences \\
\it\small ul. \'Sniadeckich 8, 00-656 Warsaw, Poland\/{\rm ;}
\it\small T.Adamowicz@impan.pl
\\
\\
Micha{\l} Gaczkowski, Przemys{\l}aw G\'{o}rka{\small{$^1$}}
\\
\it\small Department of Mathematics and Information Sciences,
\it\small Warsaw University of Technology,\\
\it\small Ul. Koszykowa 75, 00-662 Warsaw, Poland\/{\rm ;}
\it\small M.Gaczkowski@mini.pw.edu.pl, P.Gorka@mini.pw.edu.pl
}

\date{}
\maketitle

\footnotetext[1]{T. Adamowicz and P. G\'{o}rka were supported by a grant Iuventus Plus of the Ministry of Science and Higher Education of the Republic of Poland, Nr 0009/IP3/2015/73.}

\begin{abstract}
 We introduce and study strongly and weakly harmonic functions on metric measure spaces defined via the mean value property holding for all and, respectively, for some radii of balls at every point of the underlying domain. Among properties of such functions we investigate various types of Harnack estimates on balls and compact sets, weak and strong maximum principles, comparison principles, the H\"older and the Lipshitz estimates and some differentiability properties. The latter one is based on the notion of a weak upper gradient. The Dirichlet problem for functions satisfying the mean value property is studied via the dynamical programming method related to stochastic games. We employ the Perron method to construct a harmonic function with continuous boundary data. Finally, we discuss and prove the Liouville type theorems.

 Our results are obtained for various types of measures: continuous with respect to a metric, doubling, uniform, measures satisfying the annular decay condition. Relations between such measures are presented as well. The presentation is illustrated by examples.
\end{abstract}

\section{Introduction}

 Harmonic functions and the related Dirichlet problem are one of the most classical and fundamental subjects of studies in mathematical analysis and theory of PDEs. One studies harmonic functions and their generalizations in various settings and contexts, for instance in the Euclidean domains, on manifolds, in the setting of trees, also in Banach spaces. Recent two decades have been the period of an intensive development of yet another area of mathematics, the analysis on metric measure spaces. Its studies bring new approaches and sheds new light also on the theory of harmonic functions. The results due to Cheeger~\cite{cheeg}, Haj\l asz~\cite{haj}, Heinonen--Koskela~\cite{hk} and Shanmugalingam~\cite{sh}, to mention just few mathematicians contributing to the growth of the analysis on metric spaces, laid foundations for the first order Calculus and notions of gradient in metric spaces. See, for instance,  a survey by Heinonen~\cite{he07} for the panorama of the area and further references. Basing on the notion of the weak upper gradient one can study the minima of the Dirichlet energy obtaining counterparts of $p$-harmonic functions and mappings in the metric setting with the harmonic case corresponding to $p=2$, see e.g. Shanmugalingam~\cite{Sh-harm}. Related is an approach based on the Cheeger derivative and a metric counterpart of  the tangent space, see \cite{cheeg}.

 In this work we present another approach to harmonicity on metric measure spaces based on functions which satisfy the mean value property for all balls centered at the points of the given open set and contained in this set. Harmonic functions of such kind were introduced by Gaczkowski and G\'orka in~\cite{GG}. Namely, in \cite{GG} the authors study locally integrable real-valued functions defined on an open subset $\Om\subset X$ of a metric measure space
 $(X, d, \mu)$ requiring that the mean value property for $f$ holds at every $x\in \Om$ and for all balls $B(x,r)\Subset \Om$:
 \[
  f(x)=\frac{1}{\muballxr}\int_{B(x,r)} f(z) d\mu(z).
  \]
 In this work we call such functions \emph{strongly harmonic}, see Definition~\ref{defn-harm}. We continue investigations of their properties, as well as introduce the so-called \emph{weakly harmonic} functions which are required to satisfy the mean value property only for at least one admissible radius at every point of an open set. Our definition is motivated by the classical and subtle investigations in the Euclidean setting due to e.g. Koebe, Volterra and Kellogg, Hansen and Nadirashvili, and Blaschke, Privaloff and Zaremba. We refer to Section~\ref{Sec-harm} for a brief historical sketch of the studies on the size of the set of admissible radii sufficient to imply the harmonicity.

 In Preliminaries we introduce and recall some basic definitions of the metric analysis. In particular, we define continuity of a measure with respect to a metric, see Definition~\ref{def-m-cont-m}. Such a property has been important in the previous studies of harmonic functions, see \cite{GG} (also \cite{gork}). Moreover, we study some properties of a measure implying its continuity with respect to the given metric and notice that this condition gives us wide class of metric measure spaces. It turns out, for instance, that doubling measures in geodesic spaces have this property, see Proposition~\ref{prop-dbl}. Our studies involve various other types of measures, e.g. uniform measures and measures satisfying $\delta$-annular decay condition for some $\delta\in(0,1]$. However, measures continuous with respect to a distance appear to be the most general among the aforementioned measures (see the discussion and the diagram following Proposition~\ref{prop-dbl} in Preliminaries).

 In Section~\ref{Sec-harm} we bring on stage main characters of the paper, i.e. strongly and weakly harmonic functions, motivate their definitions and introduce some of their basic properties and natural relatives such as sub- and superharmonic functions. The latter two notions will play a vital role in the studies of the Dirichlet problem and the Perron method in Section~\ref{Section5}. Furthermore, we study how to generate new sub- and superharmonic functions from the existing ones.

 The key geometric and regularity properties of harmonic functions are presented and studied in Sections~\ref{sect41} and~\ref{sect4}. We identify conditions implying continuity of strongly and weakly harmonic functions as well as we discuss various Harnack inequalities on balls and compact sets. This allows us to obtain important tools of the potential and geometric function theories, namely the weak and strong maximum principles, and the comparison principle. While for strongly harmonic functions such properties could be expected, the fact that they are also available in the setting of weakly harmonic functions might be surprising. Having established the aforementioned properties of harmonic functions, we show one of the main results of Section~\ref{sect41}, namely the local H\"older continuity of strongly and weakly harmonic functions (i.e. H\"older continuity on compact subsets of an underlying domain), see Theorem~\ref{thm-holder}. For strongly harmonic functions we prove this result for geodesic doubling measure spaces with the H\"older exponent depending on the doubling constant only, whereas for weakly harmonic functions we, additionally, require that a compact set $K$ remains enough away from the boundary of the domain and the admissible radii for points in $K$ are uniformly separated from zero and uniformly bounded from the above. The final topic studied in Section~\ref{sect41} is the H\"older and Lipschitz regularity of harmonic functions in spaces satisfying the so-called $\delta$-annular decay property for some $\delta\in(0,1]$ with the Lipschitz case corresponding to $\delta=1$, see e.g. Buckley~\cite{buc}. Roughly speaking, such a property relates the measure of an annular ring to its thinness, cf. Definition~\ref{defn-an-dec}. Moreover, it turns out that length doubling metric spaces have an annular decay property for some $\delta$ while a space with measure continuous with respect to a distance possesses $1$-annular decay property. In Theorem~\ref{thm-ann-dec} we provide the H\"older and Lipschitz estimates on balls and compact sets. By imposing stronger assumptions on measure than in Theorem~\ref{thm-holder} we are able to obtain finer estimates on balls already in the H\"older case, while on compact subsets we not only show the H\"older regularity as in Theorem~\ref{thm-holder} but also provide estimates with explicit constants and exponent $\delta$. In the Lipschitz case not covered by Theorem~\ref{thm-holder} we also have explicit constants, however our estimates depend additionally on a Lebesgue number of a choosen covering.

 We continue studies of the regularity properties of harmonic functions in Section~\ref{sect4}. There, we prove Lipschitz estimates in large scale, i.e. under assumptions that points are enough far away from each other. Then, we study uniform measures, i.e. such measures $\mu$ that every ball $B$ of radius $r>0$ satisfies
 \[
  \mu(B)=C\,r^Q,\quad \hbox{for given } C>0 \hbox{ and } Q\geq 1.
 \]
 Uniform measures appear in geometric measure theory, for example in relation to the Marstrand theorem, in the studies of rectifiable measures and in the theory of incompressible flows in PDEs, see the discussion following Definition~\ref{def-uni-m}. Proposition~\ref{prop-unif-meas} shows that in spaces with uniform measures strongly and weakly harmonic functions are locally Lipschitz. Moreover, we compute Lipschitz constants more accurately than in Theorem~\ref{thm-ann-dec}. In particular we avoid using a Lebesgue number of a covering. These observations allow us to complete the presentation in Sections~\ref{sect41} and~\ref{sect4} with differentiability results based on Cheeger's work~\cite{cheeg}. In Corollaries~\ref{cor1-cheeger-style} and~\ref{cor2-cheeger-style} we describe conditions on metric measure spaces implying that strongly and weakly harmonic functions have weak upper gradients. The $(1,p)$-Poincar\'e inequality plays a crucial role for such results to hold.

 Section~\ref{Section5} is entirely devoted to studying the Dirichlet problem for harmonic functions and related Perron method. We address the following fundamental problems: whether there exists a function with given boundary data satisfying the mean value property inside the domain and whether it is unique, and if so how to construct such a function? In order to solve the first problem we take an approach based on the dynamical programming principle. The idea of this method originates from the stochastic games, especially tug-of-war games and related $p$-harmonious functions (see e.g.  Manfredi, Parviainen and Rossi~\cite{mpr}), and is based on setting up an integral operator, iterating it and proving that such an iteration process converges to a function. We adapt method by Luiro, Parviainen and Saksman~\cite{lps}, recently developed for the Euclidean domains, in the metric setting. According to our best knowledge such an approach to the Dirichlet problem on metric spaces is new in the literature. In Theorem~\ref{eps-exist-meas} we show that given a domain and a measurable boundary data one obtains a function which satisfies the mean value property with respect to exactly one radius at every point of the domain, provided that this point is enough far away from the boundary. Moreover, such a solution satisfies the boundary data condition. Furthermore, Theorem~\ref{eps-exist-cont} extends the previous result to the setting of continuous boundary data. We also prove that if a Dirichlet problem has a continuous subharmonic solution, then it has the weakly harmonic solution with the same continuous boundary data, see Theorem~\ref{sub-Dir}. The Dirichlet problem for strongly harmonic functions is solved via Perron method in Theorems~\ref{glowne} and~\ref{glowne-2}. There, we not only solve the boundary value problem for continuous data, but also show the relation between existence of barriers and regularity of boundary points. Similar relations are well-known e.g. in the setting of Newtonian harmonic functions, see the discussion following Definition~\ref{def-barr}.

 In the final section of the paper we discuss another fundamental geometric properties of harmonic functions, namely the Liouville type theorems. In Theorem~\ref{theoremLP} we provide a fairly general condition for a measure which implies that a strongly(weakly) bounded harmonic function defined in the whole space must be constant. Furthermore, we discuss some sufficient conditions on measure to guarantee that the hypotheses of Theorem~\ref{theoremLP} is satisfied. In particular, this is the case if the measure of the space is finite or in the length spaces with a doubling measure. Our discussion is illustrated with examples. We, for instance, show that even in a simple case of $\R$ there exist non-Lebesgue measures for which bounded entire harmonic functions need not be constant.

 \section{Preliminaries}

Let $(X, d, \mu)$ be a metric measure space equipped with a metric $d$ and measure $\mu$. A ball in space $X$ is denoted by $B:=B(x, r)$ for $x\in X$ and a radius $r>0$. In what follows we will assume that $\mu$ is a Borel regular measure with $0<\mu(B)<\infty$ for any ball $B\subset X$. Moreover, we assume that $X$ is proper, that is closed bounded subsets of $X$ are compact.

We say that a measure $\mu$ is \emph{doubling} if there is a constant $C_\mu>0$
such that for all balls $B=B(x,r)=\{y \in X : d(x,y)<r\}$,
\[
   \mu(2B)\le C_\mu \mu(B),
\]
where $2B(x,r)=B(x,2r)$. If $\mu$ is doubling, then $X$ is
complete if and only if it is proper (i.e.\ every closed bounded set
is compact), see Proposition~3.1 in Bj\"orn--Bj\"orn~\cite{bb}.

One of the consequences of doubling property of $\mu$ is that there exist $C, Q>0$ such that for all $x\in X$, $0<r\le R$ and $y\in B(x,R)$,
\begin{equation}\label{dbl-conseq}
  \frac{\mu(B(y,r))}{\mu(B(x,R))} \ge  \frac{1}{C} \left(\frac{r}{R}\right)^Q.
\end{equation}
In fact, $Q=\log_2 C_\mu$ and $C=C_\mu^2$ will do, see Lemma~3.3 in Bj\"orn--Bj\"orn~\cite{bb},
but there may exist a better, that is smaller, exponent $Q$. Moreover, we note that \eqref{dbl-conseq} implies that
$\mu$ is doubling, i.e.~$\mu$ is doubling if and only if there is an exponent $Q$ such that \eqref{dbl-conseq} holds.

  Furthermore, in what follows we will often appeal, without mentioning it explicitly, to the following property of doubling measures. If $(X, d, \mu)$ is a doubling metric measure space and $\Om\subset X$ is bounded with $\mu(\Om)>0$ (e.g. $\Om$ is a domain), then for any $x\in \Om$ and $0<r<\diam \Om$ it holds that
\[
  \frac{\mu(B(x,r))}{\mu(\Om)} \ge  \frac{1}{C} \left(\frac{r}{\diam \Om}\right)^{\log_2 C_\mu}.
\]

We say that $X$ is \emph{Ahlfors $Q$-regular} if there is a constant $C$
such that
\[
   \frac{1}{C} r^{Q} \le \mu(B(x,r)) \le C r^{Q}
\]
for all balls $B(x,r) \subset X$ with $r < 2 \diam X$. (In this case, the optimal choice for $q$
in \eqref{dbl-conseq} is to let $q=Q$). If we only require the left hand side of the inequality to hold, then we say that $X$ is \emph{lower $Q$-Ahlfors regular}.

One of the important properties of the metric spaces considered in the paper is the following relation between the metric and the measure.

Recall that $A \Delta B$ stands for a symmetric difference of sets $A, B\subset X$ and is defined as follows:
\[
 A \Delta B:= (A\setminus B) \cup (B\setminus A).
\]
\begin{definition}[cf. Definition 2.2 in \cite{GG}]\label{def-m-cont-m}
 Let $(X, d, \mu)$ be a metric measure space. We say that a measure $\mu$ \emph{is continuous with respect to metric $d$} if for all $x\in X$ and all $r>0$ it holds that
 \begin{equation}\label{m-cont-m}
  \lim_{\Om \ni y\underset{d}\to x} \mu(B(x,r) \Delta B(y,r))=0.
 \end{equation}
 The measure $\mu$ is called \emph{metrically continuous}.
\end{definition}

According to our best knowledge the above notion appeared for the first time in the literature in G\'orka~\cite{gork}.

The following lemma collects some basic facts about continuity of a measure with respect to the metric (see \cite{GG} for the proofs). In the presentation below we will appeal to these properties a number of times and, therefore, for the sake of convenience we present them here.
\begin{lemma}\label{meas-cont-lem}
 Let $(X, d, \mu)$ be a metric space with a Borel regular measure $\mu$. Then the following hold:
 \begin{enumerate}
 \item \label{mclem-1} If $\mu$ is continuous with respect to the metric $d$, then the map $x\to \muballxr$ is continuous in $d$.
 \item \label{mclem-2} If for every $x\in X$ and every $r>0$ it holds that $\mu(\partial B(x,r))=0$, then $\mu$ is continuous with respect to the metric $d$.
 \item \label{mclem-3} If for every $x\in X$ the function $r\to \muballxr$ is continuous, then $\mu$ is continuous with respect to the metric $d$.
 \end{enumerate}
\end{lemma}
\begin{proof}
 For the proof of Property~\ref{mclem-1}, see Corollary 2.1 in~\cite{GG}. Property~\ref{mclem-2} is proved in Lemma 2.1 in \cite{GG}, while Property~\ref{mclem-3} is proved in Theorem 2.1 in \cite{GG}.
\end{proof}

Following \cite{GG}, we recall that a metric space $(X, d)$ has the \emph{segment property} if for any
$x, y \in X$ there exists a continuous curve $\gamma:[0, 1]\to X$ joining $x$ and $y$ and such that for all $t\in [0,1]$ we have that
\[
 d(x,y)= d(x,\gamma(t))+ d(\gamma(t),y).
\]
Recall further, that a metric space $(X, d)$ is geodesic if any two points $x, y\in X$ can be joint by a curve $\gamma$ whose length equals distance $d(x,y)$. For a large class of metric spaces we can easily show their  bi-Lipschitz equivalence to geodesic spaces. Namely, let $X$ be a Loewner Ahlfors regular space (see Definition 3.1 in Heinonen--Koskela~\cite{hk}). Then $X$ is quasiconvex, see Theorem 8.23 in Heinonen~\cite{hei01}. If $X$ is additionally proper, then one can introduce a new metric in $X$ by taking the infimum of lengths of all rectifiable curves joining two points, see Remark 9.11 and Chapter 8 in \cite{hei01}, also \cite{hk} for further discussion on Loewner spaces. According to Theorem 2.2 in \cite{GG}, if $(X, d, \mu)$ is a doubling measure space with the segment property, then $\mu$ is continuous with respect to metric $d$. In a consequence we get the following result.
\begin{prop}\label{prop-dbl}
  Let $(X, d, \mu)$ be a geodesic doubling metric measure space. Then $\mu$ is continuous with respect to metric $d$.
\end{prop}
\begin{proof}
 By Theorem 2.2 in \cite{GG} it is enough to show that a geodesic space has the segment property. Indeed, let $x,y \in X$ and let $\gamma_{xy}$ be a curve such that $d(x,y)=l(\gamma_{xy})$. Choose $t\in [0,1]$ and denote $z=\gamma(t)$. Then $l(\gamma_{xy})=l(\gamma_{xz})+l(\gamma_{zy})$. Moreover, $l(\gamma_{xz})=d(x,z)$ and $l(\gamma_{zy})=d(z,y)$. Otherwise, suppose that $l(\gamma_{xz})>d(x,z)$. Then,
 $d(x,z)+d(z,y)<l(\gamma_{xz})+l(\gamma_{zy})=d(x,y)$, contradicting the triangle inequality. Therefore, we have that
 \[
 d(x,z)+d(z,y)=l(\gamma_{xz})+l(\gamma_{zy})=l(\gamma_{xy})=d(x,y)\leq d(x,z)+d(z,y).
 \]
 Hence, $X$ has a segment property and the proof of the proposition is completed.
\end{proof}

In the paper we investigate other types of measures, for instance uniform measures (Definition~\ref{def-uni-m}) and measures satisfying $\delta$-annular decay condition for some $\delta\in(0,1]$ (Definition~\ref{defn-an-dec}). As mentioned in the Introduction, it turns out that measures continuous with respect to a metric seem to be most general among the aforementioned measures. We present the following list of relations between measures studied below.

 Denote by $(X, d, \mu)$ a metric measure space and the following properties:
  \begin{itemize}
 \item[(1)] $X$ is geodesic and $\mu$ is doubling,
 \item[(1')] $X$ is a length space and $\mu$ is doubling,
 \item[(2)] $X$ has the $\delta$-annular decay property for some $\delta\in(0,1]$,
 \item[(2')] $X$ has the $1$-annular decay property,
 \item[(3)] $\mu$ is a uniform measure,
 \item[(4)] $\mu$ is continuous with respect to $d$.
 \end{itemize}

 Then, the following inclusions hold:
 \begin{align*}
 &(1) \quad \Longrightarrow \quad (4)\qquad \hbox{(Proposition~\ref{prop-dbl})}\\
 &(1') \quad \Longrightarrow \quad (2) \quad \Longrightarrow \quad (4)\qquad \hbox{(Definition~\ref{defn-an-dec} and the discussion following it, Remark~\ref{remark-ann-dec})}\\
 &(4) \quad \notimplies \quad (2) \qquad \hbox{(Example~\ref{ex-ann-dec})}\\
 &(3) \quad \Longrightarrow \quad (2')
 \end{align*}
 
We close the preliminary part of our presentation with recalling some basic definitions and facts about the first order Calculus on metric spaces. The presented results will be employed in Section~\ref{sect4} in the studies of differentiability properties of harmonic functions in metric spaces.  For foundations of the analysis on metric spaces we refer e.g. to a book and a survey article by Heinonen~\cite{hei01}, \cite{he07}, see also Heinonen--Koskela~\cite{hk} and Haj\l asz--Koskela~\cite{HajKosk}.

We say that a property holds for $p$-a.e. rectifiable curve, if it fails only for a curve family $\Gamma$ with zero
$p$-modulus, see e.g. V\"ais\"al\"a~\cite{va} and Section 2 in~\cite{hk} for definitions and properties of the modulus of curve families in Euclidean and metric settings, respectively.

\begin{definition}
 Let $(X, d, ,\mu)$ be a metric measure space and  $f:X\to[-\infty,\infty]$. We say that a nonnegative Borel function $g_f$ on $X$ is an \emph{upper gradient} if for all nonconstant rectifiable curves $\gamma:[0,l(\gamma)]\to X$,
 parameterized by arc length $ds$, we have
 \begin{equation}\label{eq:upperGrad}
    |f(\gamma(0))-f(\gamma(l(\gamma)))|\le \int_\gamma g_f\, ds
 \end{equation}
whenever both $f(\gamma(0))$ and $f(\gamma(l(\gamma)))$ are finite, and
$\int_\gamma g_f\, ds=\infty $ otherwise.

If $g_f$ is a nonnegative measurable function on $X$ and if (\ref{eq:upperGrad}) holds for $p$-a.e.~nonconstant
rectifiable curve, then $g_f$ is called a \emph{$p$-weak upper gradient} of~$f$.
\end{definition}

Upper gradients were introduced in~\cite{hk}, whereas $p$-weak upper gradients were first defined in
Koskela--MacManus~\cite{KoMc}. A relation between those two notions follows from a result in \cite{KoMc}, where
it is also shown that a $p$-weak upper gradient of $f$ can be approximated by a sequence of upper gradients of $f$ in $L^p(X)$. Moreover, if $f$ has an upper gradient in $L^p(X)$, then it has a \emph{minimal $p$-weak upper gradient} in $L^p(X)$, see Corollary~3.7 in Shanmugalingam~\cite{Sh-harm}.

Let $p\geq 1$. We say that $X$ supports a \emph{$(1,p)$-Poincar\'e inequality} if there exist constants $C_{PI}>0$ and $\lambda \ge 1$ such that for all balls $B \subset X$ and all integrable function $f$ on $X$ and all upper gradients $g_f$ of $f$,
\begin{equation} \label{PI-ineq}
    \vint_{B} |u-u_B| \,d\mu \le C_{PI} \diam B \left( \vint_{\lambda B} g_f^{p} \,d\mu \right)^{1/p},
\end{equation}
where
\[
u_B :=\vint_B u \, d\mu :=\frac{1}{\mu(B)}\int_B u \,d\mu.
\]

\section{Harmonic functions}\label{Sec-harm}

In this section we introduce and present some elementary properties of the two fundamental notions of our work, namely \emph{weakly harmonic} and \emph{(strongly) harmonic} functions for subsets of metric measure spaces, both based on the mean value property.

 Our first definition corresponds to the most classical mean value property required to hold at every point of the underlying domain. Functions with such property will be called \emph{strongly harmonic}. However, in what follows we will often drop term \emph{strongly} and write, \emph{harmonic} functions.

\begin{definition}\label{defn-harm}
 Let $\Om\subset X$ be an open set. A locally integrable function $f:\Om\to \R$ is called \emph{(strongly) harmonic} in $\Om$ if the following inequality holds for all balls $B(x,r)\Subset \Om$ with $x\in \Om$ and $r>0$:
 \[
  f(x)=\frac{1}{\muballxr}\int_{B(x,r)} f(z) d\mu(z).
  \]
 The set of all harmonic functions in $\Om$ will be denoted $\harm(\Om, \mu)$ and $\harm(\Om)$ in case the measure is fixed.
\end{definition}

The studies of relations between the harmonicity and the mean value property in the Euclidean setting have long history. It was Gauss who, perhaps first, observed that harmonic functions posses the mean value property. The opposite question, whether one need to require mean value property to hold for all radii of balls centered at the given point has also been investigated by several mathematicians, to mention results due to Koebe, Volterra and Kellogg, Hansen and Nadirashvili, and Blaschke, Privaloff and Zaremba. We refer to Section 2 in Llorente~\cite{llo} for an interesting historical account on the mean value property and harmonicity; also to Heath~\cite{hea} for further studies on to what extent the restricted mean value property is sufficient for harmonicity in the Euclidean setting. In order to motivate Definition~\ref{defn-w-harm} below more thoroughly, let us just mention that Koebe, for instance, showed that in order for a continuous function in a domain $\Om\subset \R^n$ to be harmonic it is enough to satisfy the mean value property at every $x\in \Om$ with respect to some family of radii $r^x$ with $\inf r^x=0$. If we strengthen the assumption on function and require it to be continuous on the closure of a domain, then Volterra and Kellogg proved that one radius at every point is enough for the mean value property to imply the harmonicity. Hansen and Nadirashvili improved the previous results by substituting continuity of a function up to the boundary by its boundedness. Blaschke, Privaloff and Zaremba independently observed that an asymptotic mean value property is enough to imply the harmonicity. Their results facilitated the discovery of $p$-harmonious functions, see Manfredi--Parviainen--Rossi~\cite{mpr} for the definitions and relations between $p$-harmonious functions and stochastic tug-of-war games.

In order to provide examples of studies beyond the Euclidean framework, let us mention that the mean value property appears in the setting of differential geometry, e.g. in the studies of the so-called harmonic manifolds and related notions of horospheres and the Lichnerowicz Conjecture. Recall, that a complete Riemannian manifold $M$ is called harmonic if harmonic functions on $M$ satisfy the mean value property, see Willmore~\cite{wil}, Ranjan-Shah~\cite{rans}, also Todjihounde~\cite{tod} for further definitions and references. Furthermore, see e.g. Picardello-Woess~\cite{piw} and Zucca~\cite{zuc} for the studies of mean value property in the context of harmonic functions on trees.

Motivated by the above considerations and the literature, we introduce the following more general variant of harmonic functions on metric measure spaces.

\begin{definition}\label{defn-w-harm}
 Let $\Om\subset X$ be an open set. A locally integrable function $f:\Om\to \R$ is called \emph{weakly harmonic} in $\Om$ if for every $x\in \Om$ there exists a non-empty set of positive radii $r^x_\alpha$ for $\alpha\in I$ such that the following inequality holds for all balls $B(x, r^x_\alpha)\Subset \Om$:
 \[
  f(x)=\frac{1}{\mu(B(x, r^x_\alpha))}\int_{B(x,r^x_\alpha)} f(z) d\mu(z).
  \]

 The set of all weakly harmonic functions in $\Om$ will be denoted $\wharm(\Om, \mu)$ and $\wharm(\Om)$ in case the measure is fixed.
\end{definition}

A priori we allow set of indexes $I$ to be any  non-empty set, e.g. $I$ can be uncountable. However, in what follows we will study weakly harmonic functions under minimal assumptions, namely that at every point there is at least one admissible radii and that $I$ is at most countable.

 We denote by
 \[
  r^x_M:=\sup_{i\in\{1,2,\ldots\}} r^x_i
 \]
   and related $r^{\Om}_M:=\sup_{x\in \Om} r^x_M$. However, in the presentation below we shall write $r^{\Om}_M:=r_M$, if $\Om$ is fixed or clear from the context.

 We further remark, that if $\Om$ is a bounded domain, then $r_M\leq \diam \Om$.

In what follows also the minimal radius at the point will play a role. Namely, for any $x\in \Om$ we denote by
\begin{align*}
 r^x_m&:=\inf_{i\in\{1,2,\ldots\}} r^x_i \geq 0, \\
 r^{\Om}_m&:=\inf_{x\in \Om} r^x_m.
\end{align*}
We will often require $r^x_m>0$ or $r_m>0$.

\begin{example}
Let $(\mathbb{R}, |\cdot|, |x|dx)$ be a metric measure space equipped with the Euclidean distance and a measure $\mu$ such that $d\mu:=d\mu(x)=|x|dx$. Define a function $f:\R\to\R$ as follows
\begin{equation*}
f(x)=
\begin{cases}
\frac{1}{x} \chi_{\mathbb{R}\setminus \{0\}},\quad &x\not=0,\\
0,\quad &x=0.
\end{cases}
\end{equation*}
Then $f$ is weakly harmonic but not harmonic. Indeed, $f$ is locally integrable and by letting $y \neq 0$ and $r<|y|$, we find that
\[
 \vint_{B(y,r)} f(z) d\mu = \frac{\int_{y-r}^{y+r} \frac{1}{x}|x| dx}{\int_{y-r}^{y+r}|x| dx} = \frac{1}{y}= f(y).
\]
Moreover, for any $r>0$ we have
\[
 \vint_{B(0,r)} f(z) d\mu  =0= f(0).
\]
On the other hand, if we take $y>0$ and $r>y$, then
\[
 \mu(B(y,r))=\int_{y-r}^{y+r}|x|dx=r^2+y^2,
\]
and thus
\[
 \vint_{B(y,r)} f(z) d\mu = \frac{2y}{y^2 + r^2} \neq f(y).
\]
\end{example}

Similarly we define super- and subharmonic functions.
\begin{definition}\label{defn-sub-super}
 Let $\Om\subset X$ be an open set. A locally integrable function $f:\Om\to \R$ is called a \emph{subharmonic}(\emph{superharmonic}) in $\Om$ if the following inequality holds for all balls $B(x,r)\Subset \Om$ with $x\in \Om$ and $r>0$:
 \[
  f(x)\leq (\geq)\frac{1}{\muballxr}\int_{B(x,r)} f(z) d\mu(z)
  \]
\end{definition}
We denote $\subh(\Om, \mu)$ the set of all subharmonic functions in $\Om$ with respect to the measure $\mu$ while the set of all superharmonic functions in $\Om$ will be denoted $\superh(\Om, \mu)$. For the sake of simplicity when the measure is fixed, we will often write $\subh(\Om)$ ($\superh(\Om)$).

Similarly, we define weakly sub- and superharmonic functions, cf. Definition~\ref{defn-w-harm} and denote them by
$w\subh(\Om, \mu)$ and $w\superh(\Om, \mu)$, respectively (also $w\subh(\Om)$ and $w\superh(\Om)$, respectively).

We present now further properties of harmonic functions.

\begin{prop}\label{further-prop}
 The following properties hold:
 \begin{enumerate}
 \item \label{further-prop-1} Let $f\in \harm(\Om)$. If $m\in \R$, then $f-m\in \harm(\Om)$ and $(f-m)_+\in \superh(\Om)$.
 \item \label{further-prop-2} Let $f\in \superh(\Om)$. Let $F: f(\Om)\to \R$ be concave and increasing. Then $F\circ f$ is superharmonic. Furthermore, if $f\in \harm(\Om)$, then $F\circ f$ is superharmonic for $F$ merely concave.
 \item \label{further-prop-3} Let $f\in \subh(\Om)$. Let $F: f(\Om)\to \R$ be convex and increasing. Then
 $F\circ f$ is subharmonic. Furthermore, if $f\in \harm(\Om)$, then $F\circ f$ is subharmonic for $F$ merely convex.
 \end{enumerate}
 The analogous properties hold for weakly harmonic (sub-, and superharmonic) functions.
\end{prop}
\begin{proof}
 Denote $B:=B(x,r)\Subset \Om$ a ball centered at $x\in\Om$ with $r>0$. In order to show Property \ref{further-prop-1} we note that
 \[
  f(x)-m=\frac{1}{\muball} \int_{B}f(z) d\mu - m \frac{1}{\muball} \int_{B}1 d\mu=\frac{1}{\muball} \int_{B}(f(z)-m) d\mu.
 \]
 Similarly, we show that
 \[
  \frac{1}{\muball} \int_{B} (f(z)-m)_{+} d\mu=\frac{1}{\muball} \int_{B\cap \{f>m\}}(f(z)-m) d\mu + \int_{B\cap \{f \leq m\}} 0 d\mu  \leq (f(x)-m)_{+}.
 \]
 The Young inequality gives us Property \ref{further-prop-2}.
 \begin{equation}\label{further-prop-2-ineq}
  F(f(x))\geq F\left(\frac{1}{\muball} \int_{B}f(z) d\mu\right) \geq \frac{1}{\muball} \int_{B}F(f(z)) d\mu.
 \end{equation}
 If $f\in \harm(\Om)$, then the first inequality above becomes equality giving us the second part of Property \ref{further-prop-2}.

 The proof of Property \ref{further-prop-3} follows the same lines as the one for Property \ref{further-prop-2}. In this case inequalities in \eqref{further-prop-2-ineq} are reversed due to convexity of $F$.

 For the proofs of Properties 1-3 for weakly sub/super/harmonic functions one proceeds as above restricting the discussion only to balls with admissible radii.
\end{proof}

\section{Harnack estimates, maximum principles, H\"older and Lipschitz continuity}\label{sect41}

 In this section we show several geometric properties of strongly and weakly harmonic functions such as the Harnack inequalities on balls and compact sets, strong and weak maximum principles and comparison principles.  One of the main results of this section is the H\"older continuity of harmonic functions as in Definition~\ref{defn-harm} for geodesic metric spaces with doubling measures, Theorem~\ref{thm-holder}. Moreover, for spaces satisfying the $\delta$-annular decay condition, see Definition~\ref{defn-an-dec}, we have more accurate estimates, cf. Theorem~\ref{thm-ann-dec}. We also discuss relations between measures continuous with respect to the distance and measures in Definition~\ref{defn-an-dec}, see Remark~\ref{remark-ann-dec} and Example~\ref{ex-ann-dec}.

 First, we need to refine some results from Gaczkowski-G\'orka~\cite{GG}.

\begin{prop}[Continuity of harmonic functions]\label{harm-cont}
 Let $(X,d, \mu)$ be a metric measure space with measure $\mu$ continuous with respect to metric $d$. If $f\in \harm(\Om, \mu)$, then $f$ is continuous in $\Om$.

 The same assertion holds for $f\in \wharm(\Om, \mu)$ at points $x\in \Om$ with the following property: there exists a neighborhood $U$ of $x$ such that
 \begin{equation}\label{harm-cont-wharm}
 \{r^x_1, r^x_2,\ldots\}\cap \bigcap_{y \in U}\{r^y_1, r^y_2,\ldots\}\not=\emptyset.
 \end{equation}
\end{prop}
 In other words in Proposition~\ref{harm-cont} we require that all points $y$ in every neighborhood $U$ of a point $x$ have at least one common radius with the set of admissible radii at $x$ for a weakly harmonic function $f$. Then, $f$ is continuous at all points $x\in \Om$ where such property holds.

 We note that here we do not assume that $\mu$ is doubling.
\begin{proof}
 Suppose that $x,y\in \Om$ and fix $r>0$. Then, we have
 \begin{align}
  |f(x)-f(y)|&=\left|\frac{1}{\muballxr}\int_{B(x,r)} f(z) d\mu(z)-\frac{1}{\mu(B(y, r))}\int_{B(y,r)} f(z) d\mu(z)\right| \nonumber \\
  &=\Bigg|\frac{1}{\muballxr}\int_{B(x,r)} f(z) d\mu(z)-\frac{1}{\muballxr}\int_{B(y,r)} f(z) d\mu(z) \nonumber  \\
  &-\frac{\muballxr-\mu(B(y, r))}{\muballxr \mu(B(y, r))}\int_{B(y,r)} f(z) d\mu(z)\Bigg|\nonumber  \\
  &\leq \frac{1}{\muballxr} \int_{B(x,r) \vartriangle B(y,r)} |f(z)| d\mu(z)
  + \frac{|\muballxr-\mu(B(y, r))|}{\muballxr \mu(B(y, r))}\|f\|_{L^1(B(y,r))}.\label{harm-cont-est}
 \end{align}
 Recall that, by definition, function $f\in \harm(\Om, \mu)$ belongs to $L^1_{loc}(\Om)$. Let now $y\to x$ in metric $d$. Then, by the continuity of $\mu$ with respect to $d$ we have that $\mu(B(x,r) \vartriangle B(y,r))\to 0$. This assumption together with the absolute continuity of the Lebesgue integral with respect to the measure imply that $f(y)\to f(x)$.

 Let now $f\in \wharm(\Om)$ and $x \in \Om$ satisfy the assumptions of proposition. Thus, for any sequence $\{y_k\}_{k=1}^{\infty}$ converging to $x$ for $k\to \infty$ in metric $d$ we have that for some $i\in \N$ and all $j, k\in \N$ it holds that $r^x_i=r^{y_k}_j$. Denote such a radius by $r$. This gives us a radius common for points $x$ and all $y_k$ for which the mean value property for $f$ holds. Then, estimate \eqref{harm-cont-est} holds for all $r$ and, as previously,
 \[
  \lim_{y_k\underset{d}\to x} \mu(B(x,r) \vartriangle B(y_k,r))=0
 \]
  completing the proof of the proposition.
\end{proof}

The following observation is an immediate consequence of Proposition~\ref{harm-cont}.
\begin{corol}\label{harm-loc-bdd}
 Let $(X,d, \mu)$ be a metric measure space with measure $\mu$ continuous with respect to metric $d$. If $f\in \harm(\Om, \mu)$, then $f$ is locally bounded in $\Om$. Furthermore, $f\in \wharm(\Om, \mu)$ is locally bounded on sets $E\subset \Om$ such that every $x\in E$ satisfies assumption \eqref{harm-cont-wharm} of Proposition~\ref{harm-cont}.
\end{corol}

Next, we show that the fundamental Harnack estimate holds for both weakly and strongly harmonic functions.

\begin{lemma}[The Harnack inequality on balls, cf. Lemma 3.2 in \cite{GG}]\label{Harnack-on-balls}
 Let $X$ be a metric measure space with doubling measure $\mu$ and let $f\in \harm(\Om, \mu)$ be a nonnegative function on an open set $\Om\subset X$. Suppose that a ball $B:=B(x,r)\subset \Om$ is such that $B(x, 6r)\Subset \Om$. Then the following inequality holds
 \begin{equation}\label{est-Harnack}
  \sup_{B} f \leq C_H\inf_{B} f,
 \end{equation}
 where $C_H=C_{\mu}^3$ and $C_{\mu}$ stands for a doubling constant of $\mu$.

 Moreover, let $f\in \wharm(\Om, \mu)$ be a nonnegative function on a domain $\Om\subset X$. Suppose that a ball $B:=B(x,r)\subset B(x, 2r^B_M)\Subset \Om$ is such that
 \[
  0<r^{B}_m\leq \sup_{y \in B} r^y_m < r <3r < \inf_{y\in B} r^y_M \leq r^{B}_M.
 \]
 Then, the Harnack inequality \eqref{est-Harnack} holds with constant
 \[
  C_H=C_{\mu}^{\log_{2} \frac{5r^{B}_M}{3r^{B}_m}+1}.
  \]
 \end{lemma}

\begin{proof}
 We follow the steps of reasoning in \cite{GG} and note that by assumptions on a strongly harmonic function $f$ the following inequality holds for any $y,z \in B$
 \[
  \int_{B(y, 3r)} f d\mu \geq  \int_{B(z, r)} f d\mu.
 \]
 Hence, the harmonicity of $f$ and the doubling property of $\mu$ together with the fact that $B(y, 3r)\subset B(z, 5r)$ imply
 \begin{equation}
  f(z)\leq \frac{\mu(B(y, 3r))}{\mu(B(z, r))}f(y) \leq \frac{\mu(B(z, 5r))}{\mu(B(z, r))} f(y) \leq C_{\mu}^3 f(y).
  \label{Harn-ball-strong}
 \end{equation}

 Similarly, if $f$ is weakly harmonic, then the above approach gives us for $y,z \in B$ that
 \[
  \int_{B(y, 3r)} f d\mu \geq  \int_{B(z, r)} f d\mu\geq \int_{B(z, r^z_{i_0})} f d\mu=f(z)\mu(B(z, r^z_{i_0})).
 \]
 In the last estimate we have also used the assumption that $\sup_{y \in B} r^y_m < r$, and hence there exists an admissible radius at $z$ such that $r^z_{i_0}<r$. Moreover, it holds that
 \[
  \int_{B(y, 3r)} f d\mu \leq \int_{B(y, r^y_M)} f d\mu=f(y)\mu(B(y, r^y_M)),
 \]
 since by assumptions $3r < \inf_{y\in B} r^y_M$. By the analogy to the case of $f\in \harm(\Om)$, we obtain
 \begin{equation}
  f(z)\leq \frac{\mu(B(y, r^y_M))}{\mu(B(z, r^z_{i_0}))}f(y) \leq \frac{\mu(B(z, 5/3r^y_M))}{\mu(B(z, r^{B}_m))} f(y) \leq \frac{\mu(B(z, 5/3 r^{B}_M))}{\mu(B(z, r^{B}_m))} f(y) \leq C_{\mu}^{\log_{2} \frac{5r^{B}_M}{3r^{B}_m}+1} f(y),
  \label{Harn-ball-weak}
 \end{equation}
 since $3r<r^y_M$ and, hence, $d(y,z)+r^y_M\leq 5/3 r^y_M$. Here, we also appealed to the doubling property of $\mu$.

 Since both \eqref{Harn-ball-strong} and \eqref{Harn-ball-weak} hold for any $y, z \in B$ they hold for supremum and infimum as well resulting in the assertion of the lemma.
\end{proof}

In order to show the Harnack estimate on compact sets for weakly harmonic functions we will need the following variant of Lemma~\ref{Harnack-on-balls}.

\begin{lemma}\label{Harnack-on-balls2}
 Let $X$ be a metric measure space with doubling measure $\mu$ and let $f\in \wharm(\Om, \mu)$ be a nonnegative function on a domain $\Om\subset X$. Suppose that a ball $B:=B(x,r)\subset B(x, 2r^B_M)\Subset \Om$ is such that
 \[
  0<r^{\Om}_m\leq r \leq r^{\Om}_M<\infty.
 \]
 and both $r^{\Om}_m$ and $3r^{\Om}_M$ are admissible radii for all $y\in B$.
 Then, the Harnack estimate \eqref{est-Harnack} holds with constant
 \[
  C_H=C_{\mu}^{\log_{2} \frac{5r^{\Om}_M}{3r^{\Om}_m}+1}.
  \]
\end{lemma}

\begin{proof}
 We follow the steps of Lemma~\ref{Harnack-on-balls} and, upon notation of the lemma, we arrive at the following estimates
 \begin{align*}
  \int_{B(y, 3r)} f d\mu &\geq  \int_{B(z, r)} f d\mu\geq \int_{B(z, r^{\Om}_m)} f d\mu=f(z)\mu(B(z, r^{\Om}_m)),\\
  \int_{B(y, 3r)} f d\mu &\leq \int_{B(y, 3r^{\Om}_M)} f d\mu=f(y)\mu(B(y, 3r^{\Om}_M)).
 \end{align*}
 By combining these inequalities we obtain an analog of \eqref{Harn-ball-weak}:
 \begin{equation}
  f(z)\leq \frac{\mu(B(y, r^{\Om}_M))}{\mu(B(z, r^{\Om}_m))}f(y) \leq C_{\mu}^{\log_{2} \frac{5r^{\Om}_M}{3r^{\Om}_m}+1} f(y),
 \end{equation}
 where the final constant arises from the doubling property of measure $\mu$. From this the Harnack inequality follows immediately.
\end{proof}

As an immediate consequence we obtain the Harnack estimate on compact sets.
\begin{corol}[The Harnack inequality on compact sets, cf. Theorem 3.4 in \cite{GG}]\label{Harnack-cmpt}
 Let $X$ be a geodesic metric measure space with doubling measure $\mu$, $\Om\subset X$ be an open connected set and let $f\in \harm(\Om, \mu)$ be a nonnegative function. Then, for every compact connected $K\Subset \Om$ the following inequality holds
 \begin{equation}\label{est1-c}
  \sup_{K} f \leq C \inf_{K} f,
 \end{equation}
 where $C>0$ is a constant whose value is independent of $f$, but depends among other parameters on $C_{\mu}$, a doubling constant of $\mu$.

 Moreover, for $f\in \wharm(\Om, \mu)$ estimate \eqref{est1-c} holds provided that
 \[
  0<r^{\Om}_m\leq r \leq r^{\Om}_M<\infty \quad\hbox{and } \dist(K,\partial \Om)>2r^K_M
  \]
 and both $r^{\Om}_m$ and $3r^{\Om}_M$ are admissible radii for all $y\in K$. In such a case the Harnack constant $C=C(C_{\mu}, r^\Om_m, r^\Om_M)$.
 \end{corol}

\begin{remark}
(1) In \cite{GG} the Harnack inequality has been proved for connected sets which are not necessarily path-connected. Furthermore, here we estimate Harnack constants in terms of the doubling constants and admissible radii.

(2) Note that domain $\Om$ need not be bounded. Therefore, we assume that $r^\Om_M<\infty$ in order to ensure that the Harnack constant $C$ in \eqref{est1-c} is finite.
\end{remark}

\begin{proof}[Proof of Corollary~\ref{Harnack-cmpt}]
 The proof follows the standard reasoning and, therefore, we will present only a sketch of it. For every $x\in K$ we find a ball $B(x, r_x)$ such that $B(x, 6r_x)\Subset \Om$. The collection of such balls gives us a open cover of $K$, and by compactness of $K$ we may choose a finite subcover consisting of $N$ balls. Next, take points $x, y\in K$ and connect them by a curve $\gamma$. Indeed, since the space $X$ is geodesic and $\Om$ , Lemma 4.38 in Bj\"orn--Bj\"orn~\cite{bb} implies that any two points in $K$ can be joint by a rectifiable curve. From the collection of previously chosen $N$ balls we choose such that $x\in B_1$ and $y\in B_M$ and $B_i\cap B_{i+1}\not =\emptyset$ for all $i=1,\ldots, M$. Upon choosing points $x_i\in B_i\cap B_{i+1}$ for $i=1,\ldots, M\leq N$ and applying Lemma~\ref{Harnack-on-balls} we have
 \[
  f(x)\leq C_{\mu}^2 f(x_1)\leq \cdots \leq C_{\mu}^{2(n-1)} f(x_{n-1}) \leq C_{\mu}^{2N} f(y).
 \]
 This, together with continuity of $f$, Proposition~\ref{harm-cont} imply the assertion of the corollary with $C:=C_{\mu}^{2N}$.

 The reasoning for weakly harmonic functions is similar. We cover set $K$ with open balls $\mathcal{C}:=\{B(x, r^x)\}_{x\in K}$ such that we can apply a variant of the Harnack estimate on every $B_x$ as in Lemma~\ref{Harnack-on-balls2}. Namely, we assume that $r^x:=r^{\Om}_m$ for all $x\in K$. Moreover, we need to ensure at every $x\in K$ that a ball $B:=B(x, r^x)$ satisfies $B(x,r^x)\subset B(x, 2r^B_M)\Subset \Om$. This, follows from the condition that $\dist(K,\partial \Om)>2r^K_M$. Using compactness of $K$ we choose from the cover $\mathcal{C}$ a finite cover of $K$ by balls $\{B_i\}$ for $i=1,\ldots, n$ as in the case of strongly harmonic functions.

 The remaining part of the chaining argument stays the same as in the case of strongly harmonic functions. Observe, that for all $i=1, 2,\ldots, n$
 \[
  r^{B_i}_M\leq r^{\Om}_M \quad \hbox{and} \quad r^{B_i}_m\geq r^K_m.
 \]
 In a consequence we arrive at the following chain of estimates, cf. Lemmas~\ref{Harnack-on-balls} and~\ref{Harnack-on-balls2}:
  \[
  f(x)\leq C_{\mu}^{\log_{2} \frac{5r^{\Om}_M}{3r^{\Om}_m}+1} f(x_1)\leq \cdots \leq C_{\mu}^{\log_{2} \frac{5^{n-1}}{3^{n-1}}\left( \frac{r^{\Om}_M}{r^{\Om}_m}\cdot\,\cdots\,\cdot\frac{r^{\Om}_M}{r^{\Om}_m}\right)+n} f(x_{n-1}) \leq C_{\mu}^{ \log_{2}\left( \frac{5r^\Om_M}{3r^\Om_m}\right)^n+n} f(y).
 \]
 Hence, in this case we obtain a constant $C:=C_{\mu}^{n\left( \log_{2} \frac{5r^\Om_M}{3r^\Om_m}+1\right)}$.
\end{proof}

The Harnack inequality implies, in the usual way, the strong and weak maximum principles as well as the comparison principle. The strong maximum principle for strongly harmonic functions is proved in Gaczkowski-G\'orka~\cite[Theorem 3.1]{GG} without assumption that $\Om$ is open and $\mu$ is doubling. However, their approach is different than below and for this reason as well as for the sake of completeness we present a new proof based on the Harnack inequality.

It is perhaps surprising, but the following four results are valid also for weakly harmonic functions. In fact, in order for Proposition~\ref{strong-max-pr} and Corollaries~\ref{weak-max-pr} and \ref{comp-pr-compact} to hold for $f\in \wharm(\Om)$, it is enough that at every point of a domain $\Om$ there exists one radius $r^x$ for which the mean value property is satisfied for $f$. 

\begin{prop}[The strong maximum principle]\label{strong-max-pr}
 Let $\Om\subset X$ be open connected and $\mu$ be a doubling measure on $X$. Moreover, let $f\in \harm(\Om, \mu)$ and continuous in $\Om$. If $f$ attains its maximum in $\Om$, then $f$ is constant.
 Furthermore, the assertion holds for $f\in \wharm(\Om, \mu)$ provided that $f$ is continuous (cf. \eqref{harm-cont-wharm} in Proposition~\ref{harm-cont}).
\end{prop}

\begin{proof}
 Denote $M:=\sup_{\Om} f$ and let $\Om'=\{x\in \Om: f(x)=M\}$. By Proposition~\ref{harm-cont} harmonic functions in $\harm(\Om, \mu)$ are continuous, and hence, $\Om'$ is relatively closed in $\Om$. We will show that $\Om'$ is an open subset of $\Om$. Let $B\subset \Om$ be a ball such that $B\cap \Om'\not = \emptyset$ and $6B\Subset \Om$. Denote $g:=M-f\geq 0$ in $\Om$. Proposition~\ref{further-prop}(\ref{further-prop-1}) implies that $g\in \harm(\Om, \mu)$. By the Harnack principle, Proposition~\ref{Harnack-on-balls} and by continuity of $f$ we have that
 \[
  0\leq \sup_B g \leq C_{\mu}^3 \inf_{B} (M-f)=C_{\mu}^3(M-f(x'))=0
 \]
 for some $x'\in B\cap \Om'$. 
 Thus, in fact $B\subset \Om'$ and $\Om'$ is open. The connectedness of $\Om$ implies that $\Om'$ is the only open and relatively closed subset of $\Om$ and, hence, $\Om=\Om'$. In a consequence, $f\equiv M$ and the proof is completed in the case of strongly harmonic functions.

 If $f\in \wharm(\Om, \mu)$, then the above approach may fail. Indeed, in the previous reasoning we need to know that for a set $\Om'$, there is a ball $B\subset \Om$ such that $B\cap \Om'\not = \emptyset$ and $3B\Subset \Om$. For porous sets ensuring existence of a point $x\in \Om$ and a radii $r^x_i$ for some $i=1,2,\ldots$ may require $r^x_m=0$ which, in turn, is prevented by assumptions of the Harnack inequality, cf. Corollary~\ref{Harnack-cmpt}. Instead, we follow the approach of Theorem 3.1 in \cite{GG}.

 Let $\Om'$ be as in the previous part of the proof. Continuity assumption on $f\in \wharm(\Om)$ imply that $\Om'$ is a relatively closed subset of $\Om$.

 Moreover, let us choose any $x_0\in \Om'$ with $B(x_0, r^{x_0}_i)\Subset \Om$ for some admissible radius $r^{x_0}_i$. By the harmonicity of $f$ we have that
 \[
  \frac{1}{\mu(B(x_0, r^{x_0}_i)}\int_{B(x_0, r^{x_0}_i)}(M-f(y))d\mu(y)=0.
 \]
 Since for all $x\in \Om$ it holds that $f(x)\leq M$, we obtain that $f\equiv M$ in $B(x_0, r^{x_0}_i)$. In a consequence $\Om'$ is open and, as in the case of strongly harmonic functions, we get that $f\equiv M$ in $\Om$.
\end{proof}

Recall that a metric space $X$ is \emph{locally connected} if every neighborhood of a point $x \in X$ contains a connected neighborhood. Then, the Mazurkiewicz--Moore--Menger theorem stays that $X$ is locally pathconnected provided that it is proper metric space, see Theorem~1, pg.~254, in Kuratowski~\cite{kur}. In particular, every component of an open set is open and pathconnected, see Theorem~2, pg.~253, in \cite{kur}.

  A connected space need not be locally connected (see e.g. the topologist's sine curve). Therefore, we present two variants of the weak maximum principle, related to different connectivity assumptions on the metric space.

\begin{prop}[Weak maximum principle]
Let $\Omega$ be an open bounded set in a locally connected space $X$ and $f\in \harm(\Om, \mu)\cap C(\overline{\Om})$. Then
\[
   \sup_{\overline{\Om}} f = \sup_{\partial \Om} f.
\]
\end{prop}
\begin{proof}
Since $\overline{\Om}$ is compact and $f$ is continuous in $\overline{\Om}$, there exists $x_0 \in \overline{\Om}$ such that
\[
	\sup_{\overline{\Om}} f =f(x_0).
\]
It is enough to consider only the case that $x_0 \in \Omega$. Let us denote by $\Omega(x_0)$ the connected component of $\Om$ containing $x_0$. Since $X$ is locally connected, $\Omega (x_0)$ is open and $\partial \Omega (x_0) \subset \partial \Omega$. Hence,
by the strong maximum principle we get that $f \equiv f(x_0)$ on $\overline{\Om(x_0)}$.
\end{proof}

The weak maximum principle follows immediately from Proposition~\ref{strong-max-pr} (cf. Theorem 3.2 in \cite{GG} proved under stronger assumptions than the one below).

\begin{corol}[Weak maximum principle]\label{weak-max-pr}
  Let $\Om$ be a domain in $X$ and $f\in \harm(\Om, \mu)\cap C(\overline{\Om})$. Then
  \[
   \inf_{\partial \Om} f\leq  \inf_{\overline{\Om}} f \qquad \hbox{and} \qquad \sup_{\overline{\Om}} f\leq  \sup_{\partial \Om} f.
  \]
 
 Furthermore, the assertion holds also for $f\in \wharm(\Om, \mu)$ provided that $f\in C(\overline{\Om})$, cf. \eqref{harm-cont-wharm} in Proposition~\ref{harm-cont}.
\end{corol}
\begin{proof}
 we will show only the second inequality, the first one follows the same steps. Suppose opposite, that
 $\sup_{\overline{\Om}} f>  \sup_{\partial \Om} f$. Then the maximum of $f$ is attained in $\Om$, giving by Proposition~\ref{strong-max-pr}, that $f\equiv \sup_{\overline{\Om}} f$ contradicting the continuity assumption of $f$.

 The proof of the corollary in the case of $f\in \wharm(\Om, \mu)$ follows the above lines, since under our assumptions $f$ is continuous in $\Om$.
\end{proof}

Next we show the comparison principle for harmonic functions on domains. The result follows from Proposition~\ref{strong-max-pr}  (cf. Theorem 3.2 in \cite{GG} proved under stronger assumptions on the domain).

\begin{corol}[Comparison principle]\label{comp-pr-compact}
 Let $(X, d, \mu)$ be a metric measure space and $\Om\subsetneq X$ be a domain. Let, further, $f, g \in \harm(\Om, \mu)\cap C(\overline{\Om})$ be such that $f\geq g$ on $\partial \Om$. Then $f\geq g$ in $\Om$.

 Furthermore, the assertion holds also for $f, g\in \wharm(\Om, \mu)\cap C(\overline{\Om})$ provided that at every $x\in\Om$ the sets of admissible radii of functions $f$ and $g$ have at least one common radius.
\end{corol}

\begin{proof}
 Since both $f$ and $g$ are harmonic in $\Om$, then so is also $f-g$. Since $f\geq g$ on $\partial \Om$, then $\inf_{\partial \Om} (f-g)\geq 0$. By the Corollary~\ref{weak-max-pr} we obtain that
 \[
  0\leq \inf_{\partial \Om} (f-g) \leq \inf_{\overline{\Om}} (f-g)=\inf_{\overline{\Om}} f +\inf_{\overline{\Om}} (-g).
 \]
 From this, we obtain that $\sup_{\overline{\Om}} g\leq \inf_{\overline{\Om}} f$ and the comparison principle follows.

 The proof of the corollary in the case of $f\in \wharm(\Om, \mu)$ follows the above lines. Indeed, $f$ and $g$ are continuous in $\Om$ by assumptions and $f-g$ is weakly harmonic in $\Om$, as sets of admissible radii of functions $f$ and $g$ have a common radius at every $x\in\Om$.
\end{proof}

We are in a position to state and prove the main result of this section, local H\"older continuity of harmonic functions. The proof of this result relies on the Harnack estimate on balls and holds for strongly harmonic functions.
The iteration method used below requires that for every ball of radius $r$ one is able to apply the Harnack estimate on a ball with radius $r/t$ for some $t>4$. This, however, need not be satisfied for weakly harmonic functions in a domain $\Om$ unless $r^{\Om}_m=0$, which leads constant $C$ in Lemma~\ref{Harnack-on-balls} to be unbounded.

\begin{theorem}\label{thm-holder}
 Let $X$ be a geodesic metric space with doubling measure $\mu$ and let $f\in \harm(\Om, \mu)$ for a domain $\Om\subset X$. Then, $f$ is locally H\"older continuous with the H\"older exponent depending only on the doubling constant $C_{\mu}$.

 Moreover, weakly harmonic function $f\in \wharm(\Om, \mu)$ is locally H\"older continuous in a compact set $K$ provided that $r^K_m>0$, $r^\Om_M<\infty$ and $\dist(K,\partial \Om)>5r^K_M$. In such a case the H\"older exponent depends on $C_{\mu}, r^K_m$ and $r^\Om_M$.

\end{theorem}
\begin{proof}
 Let $B:=B(x,r)\Subset \Om$ be a ball and $f\in \harm(\Om, \mu)$. Denote
 \[
  m(r)=\inf_{B} f,\qquad M(r)=\sup_{B} f.
 \]
 Then $g:=f-m(r)\geq 0$ a.e. in $B$ and by Proposition~\ref{further-prop}(\ref{further-prop-1}) function $g$ is harmonic in $\Om$. Set $t>4$. Then by the Harnack inequality on balls (Lemma~\ref{Harnack-on-balls}) we have that
 \begin{align}
  M\left(\frac{r}{t}\right)-m(r)=\sup_{B(x, \frac{r}{t})}(f-m(r))\leq C_{\mu}^2 \inf_{B(x, \frac{r}{t})}(f-m(r))=C_{\mu}^2\left(m\left(\frac{r}{t}\right)-m(r)\right).
 \end{align}
 From this we get
 \begin{align}
  m(r)(C_{\mu}^2-1)&\leq C_{\mu}^2\left(m\left(\frac{r}{t}\right)-M\left(\frac{r}{t}\right)\right)+ (C_{\mu}^2-1) M(r)\nonumber \\
  M\left(\frac{r}{t}\right)-m\left(\frac{r}{t}\right) &\leq \frac{C_{\mu}^2-1}{C_{\mu}^2}\,(M(r)-m(r)). \label{holder-iter}
 \end{align}
 Let $y\in B(x,r)$ be such that $\frac{r}{t^{n+1}}\leq d(x,y)<\frac{r}{t^n}$ for some $n=0,1\ldots$. (Such bounds always hold for some $n\in \N$ depending on $y$.) We iterate inequality \eqref{holder-iter} and obtain the following estimate:
 \begin{align}
  |f(x)-f(y)|&\leq M\left(\frac{r}{t^n}\right)-m\left(\frac{r}{t^n}\right) \leq \left(\frac{C_{\mu}^2-1}{C_{\mu}^2}\right)^n(M(r)-m(r)) \nonumber \\
  &\leq \frac{C_{\mu}^2}{C_{\mu}^2-1} (\sup_{B} f-\inf_{B} f) \left(\frac{d(x,y)}{r}\right)^{\alpha}. \label{Holder-est1}
 \end{align}
 Indeed, set
 \begin{equation}\label{Holder-alpha}
 \alpha:=\frac{\ln \left(\frac{C_{\mu}^2}{C_{\mu}^2-1}\right)}{\ln t}>0.
 \end{equation}
 Then
 \[
  \left(\frac{C_{\mu}^2-1}{C_{\mu}^2}\right)^n=t^{-\alpha n} \quad\hbox{ and }\quad t^{-\alpha n}\leq t^{\alpha} \left(\frac{d(x,y)}{r}\right)^{\alpha}=\frac{C_{\mu}^2}{C_{\mu}^2-1}\left(\frac{d(x,y)}{r}\right)^{\alpha}.
 \]
 From this, estimate \eqref{Holder-est1} follows immediately.

 Let now $B:=B(x_0, r)$ and let $x,y \in B$ for $B$ such that $4B\Subset \Om$. We distinguish two cases.

 \noindent Case 1: $d(x,y)<r$. Then, by repeating the above discussion we obtain
 \[
 |f(x)-f(y)| \leq \frac{C_{\mu}^2}{C_{\mu}^2-1} (\sup_{4B} f-\inf_{4B} f) \left(\frac{d(x,y)}{r}\right)^{\alpha}.
 \]
 Case 2: $r\leq d(x,y)< 2r$. Then,
 \[
 |f(x)-f(y)| \leq \sup_{B} f-\inf_{B} f \leq  (\sup_{4B} f-\inf_{4B} f) \left(\frac{d(x,y)}{r}\right)^{\alpha}.
 \]
 Therefore, $f$ is locally H\"older continuous with exponent $\alpha$ as in \eqref{Holder-alpha}.

 If $f\in \wharm(\Om, \mu)$, then the above reasoning can be repeated using second parts of Lemma~\ref{Harnack-on-balls} and Corollary~\ref{Harnack-cmpt}.
 \end{proof}

We close this section with yet another H\"older and Lipschitz regularity result for harmonic functions. First, we need the following definition, cf. Section 1 in Buckley~\cite{buc}.

\begin{definition}\label{defn-an-dec}
 Let $(X, d, \mu)$ be a metric measure space with a doubling measure $\mu$. We say that $X$ satisfies the \emph{$\delta$-annular decay property} with some $\delta\in(0,1]$ if there exists $A\geq 1$ such that for all $x\in X$, $r>0$ and $\ep\in(0,1)$ it holds that
 \begin{equation}\label{eq-def-an-dec}
  \mu\left(B(x,r)\setminus B(x,r(1-\ep))\right)\leq A \ep^\delta \mu(B(x,r)).
 \end{equation}
\end{definition}
If $\delta=1$, then we say that $X$ satisfies the \emph{strong annular decay property}.

Spaces with annular decay property appear, for instance, in the context of the Hardy--Littlewood and fractional maximal operators, see Buckley~\cite{buc} and Heikkinen--Lehrb\"ack--Nuutinen--Tuominen~\cite{hlnt} respectively, parabolic De Giorgi classes, see Masson--Siljander~\cite{masil}.

Among examples of spaces with strong annular decay property let us mention geodesic metric spaces with uniform measures, $\R^n$ with the Lebesgue measure and Heisenberg groups $\mathbb{H}^n$ equipped with a left-invariant Haar measures.
Moreover, Corollary 2.2 in \cite{buc} stays that \emph{if $(X, d, \mu)$ is a length metric measure space with a doubling measure $\mu$, then $X$ has the $\delta$-annular decay property for some $\delta\in(0,1]$ with $\delta$ depending only on a doubling constant of $\mu$}. In fact, Theorem 2.1 in \cite{buc} asserts that it is enough for $(X, d)$ to be the so-called $(\alpha, \beta)$-chain space in order to conclude that $X$ has the $\delta$-annular decay property. In such a case $\delta$ depends additionally on $\alpha$ and $\beta$. We refer to the discussion in Section 2 in \cite{buc} for relations between $(\alpha, \beta)$-chain spaces and the Boman chain condition and $\mathcal{C}(\lambda, M)$-condition of Haj\l asz--Koskela~\cite{HajKosk}.

\begin{remark}\label{remark-ann-dec}
 Let us comment on relation between measures satisfying Definition~\ref{defn-an-dec} and measures continuous with respect to distance. Let $x,y\in X$ and suppose that $d(x,y)<r$ for some $r>0$. If $(X, d, \mu)$ has $\delta$-annular decay property for some $\delta\in(0,1]$, then
 \begin{equation*}
\mu(B(x,r) \bigtriangleup B(y,r)) \leq \mu(B(x,r+d(x,y))\setminus B(x,r-d(x,y)))\leq A \left(\frac{2d(x,y)}{r+d(x,y)}\right)^\delta \mu(B(x,r+d(x,y)).
 \end{equation*}
 By letting $d(x,y)\to 0$ we obtain that $\mu(B(x,r) \bigtriangleup B(y,r))\to 0$ and, hence, $\mu$ is continuous with respect to $d$.
\end{remark}

 The following example shows that the opposite relation need not hold, i.e. a measure continuous with respect to a metric may fail $\delta$-annular decay property for any $\delta\in(0,1]$.
 \begin{example}\label{ex-ann-dec}
  Let $X=\R$ with the Euclidean metric and a measure $d\mu(x)=e^{-|x|}dx$. It is easy to check that $\mu$ is continuous with respect do $d$. Namely, for any ball $B(x,r)=(x-r, x+r)$ one need to consider three cases: (1) $x+r\leq 0$, (2) $x-r\geq 0$ and (3) $-r<x<r$, depending on the position of $B(x,r)$ with respect to $0$. In all cases one gets that for any fixed $x$, the function $r\mapsto \mu(B(x,r))$ is continuous and hence, by Part (\ref{mclem-3}) of Lemma~\ref{meas-cont-lem} our claim holds true.

  We will show that for large enough $x, y>0$ and large enough radii $r<x$, $r<y$ condition \eqref{eq-def-an-dec} fails. By computations we get for any $\epsilon\in(0,1)$ that
  \begin{align*}
   &\mu(B(x,r))=\int_{x-r}^{x+r}e^{-y}dy=e^{-x}(e^r-e^{-r}),\\
   &\mu(B(x,r) \bigtriangleup B(y,r))=\int_{x-r}^{x+r}e^{-y}dy-\int_{x-r(1-\epsilon)}^{x+r(1-\epsilon)}e^{-y}dy=e^{-x}(e^r-e^{-r}-e^{r(1-\epsilon)}
   -e^{-r(1-\epsilon)}).
  \end{align*}
  Therefore, by letting $\epsilon=1/r$, we obtain the following equality
  \[
  \frac{\mu(B(x,r) \bigtriangleup B(y,r))}{\mu(B(x,r))}=1-\frac{\frac{1}{e}-e^{-2r+1}}{1-e^{-2r}}\to 1-\frac{1}{e}\quad\hbox{ for }r\to \infty.
  \]
  Hence, for large enough $x,y$ and $r$ we get that condition \eqref{eq-def-an-dec} may not be satisfied with any $A>0$ and $\delta\in(0,1]$.
 \end{example}

In next theorem we show H\"older and Lipschitz estimates on balls and compact sets, thus extending Theorem~\ref{thm-holder}. Previous assumptions on measure allow us to establish H\"older estimates for some constant and exponent, whose exact values are not determined. Here, the $\delta$-annular decay property satisfied by a measure  enables us to obtain finer estimates on balls already in the H\"older case. Moreover, on compact subsets we obtain the H\"older regularity as in Theorem~\ref{thm-holder} but, additionally, provide estimates with explicit constants and exponent $\delta$. Both for the H\"older case and the new Lipschitz one, we also have explicit constants, however a dependence on a Lebesgue number of a chosen covering comes into play. Such a dependence is removed in one of our next results, see Proposition~\ref{prop-unif-meas}.

\begin{theorem}\label{thm-ann-dec}
 Let $(X, d, \mu)$ be a doubling metric measure space with a $\delta$-annular decay property for some $\delta\in(0,1]$.

 If $\delta\in(0,1)$, then a locally bounded strongly harmonic function $f$ in a domain $\Om\subset X$ is $\delta$-H\"older continuous on every ball $B:=B(x_0, r)\subset \Om$ centered at $x_0\in \Om$ such that $3B\Subset \Om$.

 If $\delta=1$, then $f$ is locally $L$-Lipschitz continuous on every ball $B\subset \Om$ such that $3B\Subset \Om$.

 In both cases we have that
 \begin{equation}\label{est33-thm-ann}
 |f(x)-f(y)|\leq 4\cdot9^\delta\|f\|_{L^{\infty}(B(x_0, 3r))}C_{\mu}^3A\, \left(\frac{d(x,y)}{r}\right)^\delta
 \end{equation}
 for $x,y \in B(x_0,\frac{r}{2})$. Here, $C_{\mu}$ stands for a doubling constant of $\mu$ and $A$ is as in \eqref{eq-def-an-dec}.

 Let $K\Subset \Om$ and $f$ be bounded in $\Om$. Furthermore, suppose that $\eta$ is a Lebesgue number of any, but fixed, open cover of $K$. Then, we have the following estimates
 \begin{align*}
  |f(x)-f(y)|&\leq 4\cdot 9^\delta\|f\|_{L^{\infty}(\Omega)}C_{\mu}^3A\, \left(\frac{2}{\dist(K, \partial \Omega)}\right)^\delta\,d(x,y)^{\delta}\quad \hbox{ for }d(x,y) \leq \eta,\\
  |f(x)-f(y)|&\leq 2\frac{\|f\|_{L^{\infty}(\Omega)}}{\eta^{\delta}}\,\, d(x,y)^{\delta}\quad \hbox{ for }
  d(x,y)> \eta.
 \end{align*}
 \end{theorem}

\begin{remark}
Suppose that $X$ is additionally geodesic. Since geodesic space is, in particular, a length space, we retrieve from Corollary 2.2 in \cite{buc} the above Theorem~\ref{thm-holder} with some $\delta \in (0,1]$.
\end{remark}

\begin{proof}[Proof of Theorem~\ref{thm-ann-dec}]

  Let $x\in \Om$ and $B(x, r)$ be a ball such that $B(x, 2r)\Subset \Om$. Then, $\dist(B(x, r), X\setminus \Om)> r$. Choose $y\in B(x,r)$ with $d(x,y)<r/2$. By the estimate similar to the one at \eqref{harm-cont-est} we get that
 \begin{align}
  |f(x)-f(y)|&=\left|\frac{1}{\muballxr}\int_{B(x,r)} f(z) d\mu(z)-\frac{1}{\mu(B(y, r+2d(x,y)))}\int_{B(y,r+2d(x,y))} f(z) d\mu(z)\right| \nonumber \\
  &\leq \frac{1}{\muballxr}\Bigg|\int_{B(x,r)} f(z) d\mu(z)-\int_{B(y,r+2d(x,y))} f(z) d\mu(z)\Bigg| \nonumber \\
  &+ \left|\frac{\mu(B(y, r+2d(x,y)))-\muballxr}{\muballxr \mu(B(y, r+2d(x,y)))}\right|\|f\|_{L^1(B(y,r+2d(x,y)))}. \label{eq-ann-dec}
  \end{align}
 Note that
 \begin{align*}
 &B(y, r+2d(x,y))\setminus B(x, r)\subset B(y, r+2d(x,y))\setminus B(y, r'), \\
 &\mu(B(y, r+2d(x,y)))-\muballxr\leq \mu(B(y, r+2d(x,y)))-\mu(B(y, r')),
 \end{align*}
 for any positive $r'<r$. Observe further, that
 \begin{align*}
 \left|\frac{\mu(B(y, r+2d(x,y)))-\muballxr}{\muballxr \mu(B(y, r+2d(x,y)))}\right|&=\frac{1}{\muballxr}-\frac{1}{\mu(B(y, r+2d(x,y)))}
 \leq \frac{\mu(B(y, r+2d(x,y)))-\mu(B(y, r'))}{\mu(B(y, r')) \mu(B(y, r+2d(x,y)))},
 \end{align*}
 as $\mu(B(y, r'))\leq \muballxr$.

 Assume that $|f|\leq M$ in $B(x,2r)$. The above discussion together with \eqref{eq-ann-dec} imply that
 \begin{align}
  |f(x)-f(y)|&\leq \frac{1}{\muballxr}\int_{B(y,r+2d(x,y))\setminus B(y, r')} |f(z)| d\mu(z)
  + M\,\frac{\mu(B(y, r+2d(x,y)))-\mu(B(y, r'))}{\mu(B(y, r'))}. \label{est2-ann-dec}
 \end{align}
 Set $r'=r-d(x,y)$. We appeal to the $\delta$-annular decay property of $X$ for $\ep$ such that
 \[
 (r+2d(x,y))(1-\ep)=r'=r-d(x,y).
 \]
 Since $\ep=1-\frac{r-d(x,y)}{r+2d(x,y)}\leq 3\frac{d(x, y)}{r}$ we set the latter expression to be $\ep$. In a consequence estimate \eqref{est2-ann-dec} takes the following form
 \begin{align*}
  |f(x)-f(y)|&\leq \frac{M}{\muballxr} A\left(\frac{3}{r}\right)^\delta d(x,y)^\delta \mu(B(y,r+2d(x,y)))\\
  &+ \frac{M}{\mu(B(y, r-d(x,y)))} A\left(\frac{3}{r}\right)^\delta d(x,y)^\delta  \mu(B(y,r+2d(x,y)))\\
  &\leq MA\left(\frac{3}{r}\right)^\delta\, d(x,y)^\delta\left(\frac{\mu(B(y,r+2d(x,y)))}{\muballxr}+\frac{\mu(B(y,r+2d(x,y)))}{\mu(B(y, r-d(x,y)))}\right).
 \end{align*}
 Finally, we appeal to the doubling property of $\mu$ and obtain
 \begin{align*}
  \frac{\mu(B(y,r+2d(x,y)))}{\muballxr}&\leq \frac{\mu(B(y,2r)}{\muballxr}\leq C_{\mu}^2\frac{\mu(B(y,\frac{r}{2})}{\muballxr}\leq C_{\mu}^2,\quad \hbox{and} \\
  \frac{\mu(B(y,r+2d(x,y)))}{\mu(B(y, r-d(x,y)))}&\leq C_{\mu}^3.
 \end{align*}
 As a result we have that
 \begin{equation}\label{est-thm-ann}
 |f(x)-f(y)|\leq 2\cdot3^\delta\|f\|_{L^{\infty}(B(x, 2r))}C_{\mu}^3A\,\left(\frac{d(x,y)}{r}\right)^\delta
 \end{equation}
 holds for all $x,y\in B(x, r/2)$.

 Suppose now, that $x, y\in B(x_0, r/2)$ such that $2B(x_0, r)\Subset \Om$. Let us consider two cases.

 {\sc Case 1:} $d(x,y)\geq \frac{1}{2}d(x,x_0)$. Then, the estimate \eqref{est-thm-ann} together with the triangle inequality imply that
 \begin{align*}
  |f(x)-f(y)|&\leq  |f(x)-f(x_0)|+ |f(y)-f(x_0)| \\
  &\leq 2\cdot3^\delta\|f\|_{L^{\infty}(B(x_0, 2r))}C_{\mu}^3A\,\left(\left(\frac{d(x,x_0)}{r}\right)^\delta+\left(\frac{d(y,x_0)}{r}\right)^\delta\right)\\
  & \leq 2\cdot3^\delta\|f\|_{L^{\infty}(B(x_0, 2r))}C_{\mu}^3A\,\left(\left(\frac{2d(x,y)}{r}\right)^\delta+\left(\frac{d(x,x_0)+d(x,y)}{r}\right)^\delta\right)\\
  & \leq (2\cdot3^\delta)^2\|f\|_{L^{\infty}(B(x_0, 2r))}C_{\mu}^3A\,\left(\frac{d(x,y)}{r}\right)^\delta.
 \end{align*}

 {\sc Case 2:} $d(x,y)< \frac{1}{2}d(x,x_0)\leq r/2$. Then, $y\in B(x, r/2)$ and $B(x, 2r)\subset B(x_0, 3r)\Subset \Om$.
  Thus, from \eqref{est-thm-ann} we obtain
 \begin{align*}
  |f(x)-f(y)|\leq 2\cdot3^\delta\|f\|_{L^{\infty}(B(x, 2r))}C_{\mu}^3A\,\left(\frac{d(x,y)}{r}\right)^\delta.
 \end{align*}

 We note that $\|f\|_{L^{\infty}(B(x, 2r))}\leq \|f\|_{L^{\infty}(B(x_0, 3r))}$ and combine inequalities obtained in both cases to obtain the assertion of Theorem~\ref{thm-ann-dec} for balls.

Let $K$ be a compact subset of $\Omega$ and $r=\hbox{dist}(K, \partial \Omega)$. Since $K$ is compact, we can cover it by open balls $B(x, r/2)$ centered at points $x\in K$ and choose a finite subcover, denoted by $B_i:=B(x_i, r/2)$ for $i=1,2,\ldots, N$ for some $N$. Hence,
\[
	K\subset \bigcup_{i=1}^N B(x_i, r).
\]
Let us denote by $\eta$ a Lebesgue number of covering $\{B_i\}$. If $x, y\in K$ are such that $d(x,y) \leq \eta$, then there exists $i_0\in \{1,...,N\}$ such that $x, y \in B_{i_0}$. This allows us to apply \eqref{est33-thm-ann}
and obtain the following inequality
\[
|f(x)-f(y)|\leq 4\cdot 9^\delta\|f\|_{L^{\infty}(\Omega)}C_{\mu}^3A\, \left(\frac{2}{\hbox{dist}(K, \partial \Omega)}\right)^\delta\,d(x,y)^{\delta}.
\]
Otherwise, if $d(x,y)> \eta$, then we get
\[
 |f(x)-f(y)|\leq |f(x)|+|f(y)| \leq 2\|f\|_{L^{\infty}(\Omega)} \frac{d(x,y)^{\delta}}{\eta^{\delta}}.
\]
This completes the proof of the second part of Theorem~\ref{thm-ann-dec} and the whole proof is, thus, completed.
\end{proof}

\section{The Lipschitz regularity and uniform measures. Weak upper gradients of harmonic functions}\label{sect4}

In this section we study some differentiability properties of harmonic functions. We begin with showing that harmonic functions are Lipschitz in the large scale, that is assuming that the distance between points is not too small. For the similar result in the setting of maximal operators see Lemma 8 in Haj\l asz--Mal\'y~\cite{hm}. In the previous section, see Theorem~\ref{thm-ann-dec}, we show, among other results, the Lipschitz regularity for harmonic functions imposing the $1$-annular decay condition on the measure. In Proposition~\ref{harm-lip} we allow measure to be just doubling, however we require harmonic function to be nonnegative.

The second main result of this section is Proposition~\ref{prop-unif-meas}. We study the Lipschitz regularity of strong and weakly harmonic functions in the case of the uniform measure growth. Such measures play a fundamental role e.g. in the geometric measure theory, see the discussion and references below.

Finally, using the celebrated Cheeger's results on differentiability of Lipschitz functions, in Corollaries~\ref{cor1-cheeger-style} and \ref{cor2-cheeger-style} we study the existence of weak upper gradient for strongly and weakly harmonic functions.

\begin{prop}\label{harm-lip}
 Let $(X, d, \mu)$ be a metric space with doubling measure $\mu$ and let $\Om\subset X$ be a domain. Suppose that $f\in \harm(\Om, \mu)$ is nonnegative. Then the following estimate holds:
 \begin{equation}
  |f(x)-f(y)|\leq c \frac{d(x,y)}{r},
 \end{equation}
 for every $r>0$ and all $x,y \in B(x_0, r)\subset B(x_0, 2r)\Subset \Om$, such that $r\left(C^{\frac{1}{Q(C-1)}}-1\right)\leq d(x,y)$. Here constant
 \[
 c=\frac{QC^2}{\mu(B(x_0, 2r))} \|f\|_{L^1(\Om)},
 \]
 where $C$ and $Q$ are constants from estimate~\eqref{dbl-conseq}. One may assume that $Q=\log_2 C_\mu$ and $C=C_\mu^2$ (see Lemma~3.3 in Bj\"orn--Bj\"orn~\cite{bb}).
\end{prop}

We further note that the lower bound condition in the proposition imposed on a distance $d(x,y)$ is equivalent to a condition on the doubling constant of $\mu$. Namely, with the above notation and values of $Q$ and $C$ it holds that $C^{\frac{1}{Q(C-1)}}-1<1$ provided that $C_{\mu}^2>1+2/\log_23$. This can be easily satisfied by increasing $C_{\mu}$ if necessary.

\begin{proof}
  Let $x, y\in B(x_0, r)$ for some $x_0\in \Om$. The doubling property of $\mu$, cf. \eqref{dbl-conseq}, implies that for balls $B(y, r)\subset B(x, r+d(x,y))$ we have that
  \begin{equation*}
   \frac{\mu(B(y,r))}{\mu(B(x, r+d(x,y)))} \geq  \frac{1}{C} \left(\frac{r}{r+d(x,y)}\right)^Q.
  \end{equation*}

 Upon simple algebraic transformations of the assumption relating the distance $d$ and $\epsilon, C$ and $Q$ we get that
 \begin{align}
 &\left(\frac{\epsilon}{\epsilon+d(x,y)}\right)^{Q(C-1)}\leq \frac{1}{C} \nonumber \\
 &\left(\frac{\epsilon}{\epsilon+d(x,y)}\right)^{QC}\leq \frac{1}{C} \left(\frac{\epsilon}{\epsilon+d(x,y)}\right)^{Q}.
 \label{lip-est1}
 \end{align}

  We now appeal to the Bernoulli inequality, combining it with the estimate in \eqref{lip-est1}, to obtain that
  \begin{align*}
   \frac{1}{C}\left(\frac{r}{r+d(x,y)}\right)^Q = \frac{1}{C}\left(1-\frac{d(x,y)}{r+d(x,y)}\right)^{QC}\geq 1-{QC}\frac{d(x,y)}{r+d(x,y)}\geq 1-\frac{QC}{r} d(x,y).
  \end{align*}
 Hence,
 \begin{align*}
& \frac{1}{\mu(B(x, r+d(x,y)))} \int_{B(x, r+d(x,y))} f(z) d\mu(z) \geq  \frac{1}{\mu(B(x, r+d(x,y)))}
   \int_{B(y, r)} f(z) d\mu(z) \\
& = \frac{\mu(B(y, r))}{\mu(B(x, r+d(x,y)))} \vint_{B(y, r)} f(z) d\mu(z) \geq
\frac{1}{C}\left(\frac{r}{r+d(x,y)}\right)^Q \vint_{B(y, r)} f(z) d\mu(z)  \\
&\geq \left(1-\frac{QC}{r} d(x,y)\right) \vint_{B(y, r)} f(z) d\mu(z).
 \end{align*}
 This observation together with the nonnegativity and harmonicity assumptions on $f$ imply the following estimate:
 \begin{align*}
 f(x)-f(y)&=\vint_{B(x, r+d(x,y))} f(z) d\mu(z)- \vint_{B(y, r)} f(z) d\mu(z) \geq -\frac{QC}{r} d(x,y)\vint_{B(y, r)} f(z) d\mu(z)
 \end{align*}
  The above inequality holds true if we replace $x$ with $y$ and, therefore, the following inequality holds:
  \begin{equation}\label{lip-est2}
   |f(x)-f(y)|\leq \frac{QC}{r \mu(B(y, r))} \|f\|_{L^1(\Om)}d(x,y).
  \end{equation}
  We again appeal to property \eqref{dbl-conseq} of doubling measures and get
  \[
   \frac{\mu(B(y, r))}{\mu(B(x_0, 2r))}\geq \frac{C}{2^Q}.
  \]
  In a consequence,
  \begin{equation}\label{lip-est3}
   |f(x)-f(y)|\leq \frac{QC^2}{\mu(B(x_0, 2r))} \|f\|_{L^1(\Om)} \frac{d(x,y)}{r}.
  \end{equation}
 This completes the proof of Proposition~\ref{harm-lip}.
\end{proof}

The next result gives us the Lipschitz regularity in the important case of the uniform measure growth.
\begin{definition}\label{def-uni-m}
Let $(X, d, \mu)$ be a geodesic metric space equipped with a Borel regular measure $\mu$. We call $\mu$ a \emph{$Q$-uniform measure} for some $Q\geq 1$, if there exists a constant $C>0$ such that for any $x\in X$ and all $r>0$
 \begin{equation}\label{meas-cond}
  \mu(B(x,r))=Cr^Q.
 \end{equation}

\end{definition}

 Uniform measures play an important role in geometric measure theory. For instance, if $X=R^n$, then Marstrand~\cite{mar} proved that for a non-trivial $Q$-uniform measure it necessarily holds that $Q\in \mathbb{N}$ (see also Chousionis--Tyson~\cite{cht} for a discussion of Marstrand's theorem and uniform measures in the setting of Heisenberg groups). One of results of the celebrated paper by Preiss~\cite{pre} stays that for $Q=1,2$ uniform measures are flat. Let us also mention that uniform measures have been employed to investigate relations between harmonic measures and non-tangentially accessible domains (NTA-domains), see Kenig--Preiss--Toro~\cite{kpt} and in the studies of rectifiable measures, see Tolsa~\cite{tol}. Moreover, uniform measure appear in potential and stochastic analysis, see Bogdan--St\'os--Sztonyk~\cite{bss}, in the theory of incompressible flows with vorticities, see Cie\'slak--Szuma\'nska~\cite{csz}.

\begin{prop}\label{prop-unif-meas}
 Let $(X, d, \mu)$ be a geodesic metric space such with a $Q$-uniform measure $\mu$. Then, any harmonic function $f\in \harm(\Om, \mu)$ is locally $L$-Lipschitz on every compact $K\subset \Om$ for $L=Q2^{Q+1}\frac{M}{\dist(K, X\setminus \Om)}$ and $M=\|f\|_{L^{\infty}(K)}$.

 Furthermore, the assertion holds for $f\in \wharm(\Om, \mu)$ on every compact $K\subset \Om$ provided that $$r^K_m>\dist(K, X\setminus \Om)>0$$
 and the following condition holds:
 \[
 \bigcap_{x \in K}\{r^x_1, r^x_2,\ldots\}\not=\emptyset,
 \]
 that is, all points in $K$ must have at least one common radius for which the mean value property holds for $f$. Moreover, in such a case we have $L=Q2^{Q+1}\frac{M}{r^K_m}$.
\end{prop}

\begin{remark}
  It is easy to see that uniform measures satisfy $1$-annular decay property, see Definition~\ref{defn-an-dec} and so for strongly harmonic functions Proposition~\ref{prop-unif-meas} follows from the Theorem~\ref{thm-ann-dec}. However, below we are able to describe more accurately dependence of the Lipschitz constant on the parameters of the underlying space and the harmonic function. In particular, we avoid using a Lebesgue number of a covering. Moreover, the result below gives also the Lipschitz regularity for weakly harmonic functions.
\end{remark}

\begin{proof}[Proof of Proposition~\ref{prop-unif-meas}]
 Let $f\in \harm(\Om, \mu)$. Note that $\mu$ is a doubling measure with a doubling constant $C_\mu=2^{-Q}$. Then, Proposition~\ref{prop-dbl} implies that $\mu$ is continuous with respect to metric $d$. In a consequence, we infer from Corollary~\ref{harm-loc-bdd} that $f$ is locally bounded. Denote by $M$ an upper bound of $f$ on some compact set $K\subset \Om$. The estimate similar to \eqref{harm-cont-est} in Proposition~\ref{harm-cont} gives us that
 \begin{align}
   |f(x)-f(y)|&\leq \frac{1}{\muballxr} \int_{B(x,r) \vartriangle B(y,r)} |f(z)| d\mu(z)
  + \frac{|\muballxr-\mu(B(y, r))|}{\muballxr \mu(B(y, r))}\left|\int_{B(y,r)}f(z)d\mu(z)\right| \nonumber\\
  & \leq 2M\frac{\mu(B(x,r) \vartriangle B(y,r))}{\muballxr}. \label{unif-est}
\end{align}
 As in Corollary~\ref{cor-Liouv} we assume that $r>d(x,y)$. Then, by \eqref{cor-Liouv-e1} and \eqref{meas-cond}
 \begin{equation*}
\mu(B(x,r) \bigtriangleup B(y,r)) \leq \mu(B(x,r+d(x,y))\setminus B(x,r-d(x,y)))= C(r+d(x,y))^Q-C(r-d(x,y))^Q.
 \end{equation*}
Hence, by the mean value theorem applied to function $s^Q$, we obtain that
\begin{equation*}
\frac{\mu(B(x,r) \vartriangle B(y,r))}{\muballxr}\leq \left(1+\frac{d(x,y)}{r}\right)^Q-\left(1-\frac{d(x,y)}{r}\right)^Q=2Qs^{Q-1}\frac{d(x,y)}{r},
\end{equation*}
 where $s\leq 1+\frac{d(x,y)}{r}<2$. Choose $x,y\in K$ and let $r:=\frac{1}{2}\dist(K, X\setminus \Om)$. Thus, for all $x, y\in K$ it holds
 \begin{equation*}
   |f(x)-f(y)|\leq L d(x,y),
 \end{equation*}
 where $L=Q2^{Q+1}\frac{M}{\dist(K, X\setminus \Om)}$.

 If $f\in \wharm(\Om, \mu)$, then by assumptions for all $x, y\in K$ we are able to find at least one radius, denoted by $r$, such that \eqref{unif-est} and estimates following it hold for $r$. Then, in the final step we have that
 \[
 \frac{d(x,y)}{r}\leq \frac{d(x,y)}{r^K_m},
 \]
 and so in this case $L=Q2^{Q+1}\frac{M}{r^K_m}$, as desired. The proof is therefore completed.
\end{proof}

One of the consequences of Proposition~\ref{prop-unif-meas} is the differentiability of Lipschitz weakly and strongly harmonic functions on compact sets.

Let $f$ be a locally Lipschitz function in $\Om\subset X$. We define a \emph{lower pointwise dilatation} of $f$ as follows
 \[
  {\rm lip} f(x)=\liminf_{r\to 0} \sup_{y\in B(x, r)}\frac{|f(x)-f(y)|}{r}.
 \]
Similarly, we define an \emph{upper pointwise dilatation} of $f$ by the formula:
 \[
  {\rm Lip} f(x)=\limsup_{r\to 0} \sup_{y\in B(x, r)}\frac{|f(x)-f(y)|}{r}.
 \]

We refer to the following result due to Cheeger, see also Preliminaries for a discussion on weak upper gradients and the Poincar\'e inequalities.

\begin{theorem}[Theorem 6.1 in Cheeger~\cite{cheeg}]
 Let $(X, d, \mu)$ be a complete doubling metric measure space supporting $(1,p)$-Poincar\'e inequality for $p>1$. Let further $f$ be a locally Lipschitz function in a domain $\Om\subset X$. Then the minimal $p$-weak upper gradient $g_f$ of $f$ exists and
 \[
 g_f={\rm lip} f \quad \hbox{a.e. in } \Om.
 \]
 Moreover,
 \[
  {\rm lip} f={\rm Lip} f \quad \hbox{ a.e. in }\Om,
 \]
 and both ${\rm lip} f$ and ${\rm Lip} f$ are upper gradients of $f$.
\end{theorem}

One combines the above Cheeger's theorem with Proposition~\ref{prop-unif-meas} to obtain the following observation. (Recall that uniform measures are doubling.)
\begin{corol}\label{cor1-cheeger-style}
 Let $(X, d, \mu)$ be a complete geodesic metric space such with a $Q$-uniform measure $\mu$ for some $Q\geq 1$ and supporting $(1,p)$-Poincar\'e inequality for $p>1$. Suppose that $f$ is a strongly harmonic function in a compact set $K\subset \Om$. Then, the minimal $p$-weak upper gradient $g_f$ of $f$ exists and
 \[
  g_f={\rm lip} f ={\rm Lip} f
 \]
 a.e. in $\Om$.

 Furthermore, the assertion holds for a weakly harmonic function $f$ in $K$ provided that $f$ satisfies the assumptions of the second part of Proposition~\ref{prop-unif-meas}.
\end{corol}

Similarly, by combining the Cheeger's theorem with Theorem~\ref{thm-ann-dec} we arrive at the following result.
\begin{corol}\label{cor2-cheeger-style}
 Let $(X, d, \mu)$ be a complete doubling metric measure space with a $1$-annular decay property supporting $(1,p)$-Poincar\'e inequality for $p>1$. Suppose that $f$ is a strongly harmonic function in a ball $B \subset 2B \Subset \Om$. Then, the minimal $p$-weak upper gradient $g_f$ of $f$ exists and $g_f={\rm lip} f ={\rm Lip} f$ a.e. in $B$.
\end{corol}

\section{The Dirichlet problem and Perron solutions}\label{Section5}

We begin this section with an observation that in general, the Dirichlet boundary value problem need not have a solution even in the simplest one-dimensional case as the following example shows.
\begin{example}
Let $f \in \mathcal{H}( \mathbb{R}, |\cdot|, |x| dx)$ and set $g:=f|_{(0,1)}$. Then  $g\in \mathcal{H}((0, 1), |\cdot|, xdx)$ and by observing that $x \in \mathcal{H}( (0, 1), |\cdot|, dx)$, we conclude by Proposition~\ref{Liouv-prop1}
that $xg\in \mathcal{H}( (0, 1), |\cdot|, dx)$ and, thus, $g(x) = \frac{A}{x} + B$ for some positive constants $A, B$ and, hence, $f \equiv B$ by the continuity of $f$ (see also Example~\ref{ex-Liouv2} below).

Consider the Dirichlet problem of finding a harmonic function $g$ in $\mathcal{H}((0,1), |\cdot|, xdx)$ such that $g(0)\not=g(1)$. Then, by the above reasoning there is no solution of such problem.
\end{example}

The purpose of this section is to study the following questions:
\smallskip

\begin{quote} (1) \emph{When does a Dirichlet problem for a functions with the mean value property as in Definition~\ref{defn-harm} have a solution and for what type of boundary data?}
\smallskip

\noindent (2) \emph{How to construct a solution to the harmonic Dirichlet problem?}
\end{quote}

Although these questions are nowadays classical in the Euclidean setting, see e.g. Gilbarg--Trudinger~\cite{gt}, their metric counterparts have been intensively studied mainly in past two decades, see e.g. Section 10 in Bj\"orn--Bj\"orn~\cite{bb} and references therein. However, results in \cite{bb} apply to harmonic functions defined as minimizers of the $2$-Dirichlet energy, whereas we answer the above two questions for harmonic functions defined via the mean value property.

In the first part we study the solvability of the Dirichlet problem by employing the so-called Dynamical programming principle, generalized to the metric setting and based on studies conducted for $p$-harmonious functions in Euclidean domains by Luiro--Parviainen--Saksman~\cite{lps} and Manfredi--Parviainen--Rossi~\cite{mpr}. Our results apply to some functions with the mean value property for measurable and continuous boundary data. In the second part we employ the Perron method and solve the Dirichlet boundary problem for strongly harmonic functions with continuous data.

\subsection{The Dirichlet problem and the Dynamical programming principle}

Our approach to the Dirichlet problem for harmonic functions in metric space is the following. First, we extend the dynamical programming principle as presented in~\cite{lps} to the setting of metric spaces with Borel regular measures. Theorems~\ref{eps-exist-meas} and \ref{eps-exist-cont} below extend Theorems 2.1 and 4.1 in \cite{lps} for measurable and continuous data, respectively. Moreover, Theorem~\ref{eps-exist-meas} provides us with a function $u_{\epsilon}$, the solution to the Dirichlet problem with measurable data, such that it satisfies the mean value property on balls with radii $\epsilon$. In Theorem~\ref{eps-exist-cont} we show that the similar property holds for
solutions to the Dirichlet problem with continuous data at points with $\ep$-distance from the complement of the domain. Finally, we show that in metric measure spaces with the $\delta$-annular decay property the existence of subharmonic solution of the boundary value problem with continuous data implies existence of the weakly harmonic continuous function with the same continuous boundary data, see Theorem~\ref{sub-Dir}.

We follow the notation of Section 2.1 in \cite{lps} and for $\ep>0$ define the $\ep$-boundary strip of a domain $\Om\subset X$:
\[
\Gamma_\ep=\{x\in X\setminus \Om: \dist(x, \Om)\leq \ep \}.
\]
In order to ensure that $\Gamma_\ep\not=\emptyset$ we need to assume that $\Om$ is such that $\Gamma_\ep\subset X$. This assumption will not weaken our results, as in fact we apply them only for balls $B_\ep$ compactly contained in the underlying domain with radii $\ep$ small enough so that $B_\ep\cup \Gamma_\ep$ remains a subset of the domain.

Denote by $\Om_\ep:=\Om\cup \Gamma_\ep$.

\begin{theorem}\label{eps-exist-meas}
 Let $(X,d, \mu)$ be a metric measure spaces with a measure continuous with respect do $d$. Moreover, let $g:\partial \Om\to \R$ be a bounded Borel measurable function and $F:\Gamma_\ep\to \R$ be defined as follows:
 \begin{equation}
  F(x)=
  \begin{cases}
   &g(x),\quad x\in \partial \Om, \\
   &0,\quad x\in \Gamma_\ep\setminus \partial \Om.
  \end{cases}
 \end{equation}
 Then, there exists a bounded Borel function $u:\Om_\ep\to \R$ solving the following Dirichlet problem with the boundary data $F$:
 \begin{equation}\label{eps-exist-est}
  \begin{cases}
   &u(x)=\vint_{B(x,\ep)} u(y)d\mu(y),\quad x\in \Om,\\
   &u|_{\Gamma_\ep}=F.
  \end{cases}
 \end{equation}
 In fact, $u$ is the uniform limit of a sequence $\{u_i\}_{i=0}^{\infty}$ defined via the following iteration scheme:
 \begin{equation}\label{def-uzero}
 u_0(x)=
  \begin{cases}
   &\inf_{x\in \Gamma_\ep} F(x),\quad x\in \Om,\\
   &F(x),\quad x\in \Gamma_\ep,
  \end{cases}
 \end{equation}
 while for $i=0,1,\ldots$ we define
 \begin{equation}\label{def-ui}
 u_{i+1}(x)=Tu_i(x):=
  \begin{cases}
   &\vint_{B(x,\ep)} u_i(y)d\mu(y),\quad x\in \Om,\\
   &u_i(x),\quad x\in \Gamma_\ep.
  \end{cases}
 \end{equation}
\end{theorem}

\begin{proof}
 First, notice that $Tu_i$ is Borel bounded measurable function in $\Om_\ep$ for every $i=0,1,\ldots$. In order to verify this observation, first notice that $u_0$ is measurable. Indeed, from \eqref{def-uzero} we infer that
 \begin{equation*}
 u_0(x)=
  \begin{cases}
   &\min\{\inf_{\partial \Om}g, 0\} ,\quad x\in \Om,\\
   &g(x),\quad x\in \partial \Om,\\
   &0,\quad x\in \Gamma_\ep\setminus \partial \Om.
  \end{cases}
 \end{equation*}
  This, together with Borel measurability of $g$ imply that $u_0$ is Borel measurable as well. Next we compute
 \begin{align*}
 u_1(x):= Tu_0(x)&=
  \begin{cases}
   &\vint_{B(x,\ep)} u_0(y)d\mu(y),\quad x\in \Om,\\
   &u_0(x),\quad x\in \Gamma_\ep.
  \end{cases} \\
  &=
  \begin{cases}
   &\frac{\mu(B(x,\ep)\cap \Om)}{\mu(B(x, \ep))}\min\{\inf_{\partial \Om}g, 0\} ,\quad x\in \Om,\\
   &g(x),\quad x\in \partial \Om,\\
   &0,\quad x\in \Gamma_\ep\setminus \partial \Om.
  \end{cases}
 \end{align*}
  The continuity of $\mu$ with respect to $d$ gives us that $u_1(x)$ is continuous for $x\in \Om$, whereas for $x\in\Om_{\ep}\setminus \Om$ the measurability of $u_1$ follows the same argument as for $u_0$. Indeed, by Lemma~\ref{meas-cont-lem}(\ref{mclem-1}) we have that a function $x\mapsto \mu(B(x,\ep))$ is continuous in $\Om$.
  Similarly, a function $\Om\ni x\to \vint_{B(x,\ep)} u_0(y) d\mu(y)$ is measurable as a quotient of two measurable functions. The measurability of $u_{i+1}=Tu_i$ for $i=1,2,\ldots$ follows the same steps by induction and we omit details.

 We continue the proof following the reasoning for Theorem 2.1 in \cite{lps}. It is immediate to check that $u_1\geq u_0$ in $\Om_\ep$. In order to give an idea about formulas describing functions $u_i$ we compute also
 \begin{align*}
 u_2(x)&:=Tu_1(x) \\
 &=
  \begin{cases}
   &\frac{\mu(B(x,\ep)\cap \{y\in \Om\,:\,\dist(y, \partial \Om)\geq \ep\})}{\mu(B(x, \ep))}+
   \frac{\min\{\inf_{\partial \Om}g, 0\}}{\mu(B(x, \ep))}\int_{B(x,\ep)\cap \{y\in \Om\,:\,\dist(y, \partial \Om)< \ep\}}\frac{\mu(B(y, \ep)\cap \Om)}{\mu(B(y, \ep))}d\mu(y) ,\quad x\in \Om,\\
   &g(x),\quad x\in \partial \Om,\\
   &0,\quad x\in \Gamma_\ep\setminus \partial \Om.
  \end{cases}
 \end{align*}
 This yields
 \[
  u_2=Tu_1\geq Tu_0=u_1.
 \]
 The same reasoning allows us to conclude that $u_{i}\leq u_{i+1}$ in $\Om_\ep$ for all $i=0,1,\ldots$. Definitions of $u_0$ and operator $T$ together with easy argument by induction imply that the sequence $\{u_i\}_{i=0}^{\infty}$ is increasing and uniformly bounded from above by $\sup_{\partial \Om} g<\infty$. The latter property is a consequence of a simple induction applied with \eqref{def-ui}. Namely, since $|u_0|\leq \sup_{\partial \Om}g$ in $\Om_{\ep}$, then so is $u_1$. Then, by assuming that $|u_i|\leq \sup_{\partial \Om}g$ in $\Om_{\ep}$ for some $i>1$ we trivially obtain that for $x\in \Om$
 \[
 |u_{i+1}(x)|\leq \vint_{B(x,\ep)} |u_i(y)|d\mu(y)\leq \sup_{\partial \Om}g,
 \]
 while otherwise, in $\Om_{\ep}\setminus \Om$, the boundedness of $u_{i+1}$ immediately follows from its definition.

 Therefore, we are in a position to define the following bounded Borel measurable function $u:\Om_\ep\to \R$:
 \[
  u(x)=\lim_{i\to\infty}u_i(x).
 \]
  Next, one shows that the convergence is uniform and the proof is by contradiction. This proof follows steps of the corresponding proof of Theorem 2.1 in~\cite{lps} for $\alpha=0$ and $\beta=1$ and, therefore, will be omitted. Let us comment, that the uniform convergence implies that $u$ satisfies \eqref{eps-exist-est} and, by construction, has the boundary data $F$.
 \end{proof}

In the next theorem we prove the solvability of the Dirichlet problem similar to \eqref{eps-exist-est} for continuous boundary data.

Define
\[
 \Gamma_{\ep,\ep}:=\{x\in X\,:\,\dist(x,\partial \Om)\leq \ep\},\quad\Om_{\ep}:=\Om\cup\Gamma_{\ep,\ep}
 \quad\hbox{and}\quad \Gamma_{\ep}:=\Gamma_{\ep,\ep}\setminus \Om.
\]
Using the Tietze extension theorem applied to $\partial \Om$ and a continuous function $g:\partial \Om\to \R$ we obtain a continuous function $F:\Gamma_{\ep,\ep}\to \R$ such that $F|_{\partial \Om}=g$. Moreover,
\[
 \sup_{x\in \Gamma_{\ep,\ep}}|F(x)|= \sup_{x\in \partial \Om}|g(x)|<\infty.
\]

Let us remark, that if $F$ is an extension of $g$ for some $\ep_1$, then $F$ can be taken also for all $\ep\leq \ep_1$.
\begin{theorem}[cf. Theorem 4.1 in \cite{lps}]\label{eps-exist-cont}
 Let $\Om$ be a locally connected domain in $(X, d, \mu)$ with $\mu$ continuous with respect to $d$ and let $F:\Gamma_{\ep,\ep}\to \R$ be a continuous function as above. Then, there exists a unique continuous $u_\ep:\Om_{\ep}\to\R$ which solves the following boundary value problem:
\begin{equation}\label{eps-exist-cont-est}
 u_\ep(x)=\left\{
  \begin{array}{cccc}
   &\vint_{B(x,\ep)} u_\ep(y)d\mu(y),&\, &x\in \Om\setminus \Gamma_{\ep,\ep}\\
   &\left(1-\frac{\dist(x,\Gamma_{\ep})}{\ep}\right)F(x)+\frac{\dist(x,\Gamma_{\ep})}{\ep}\vint_{B(x,\ep)} u_\ep(y)d\mu(y),&\,&x\in \Om\cap\Gamma_{\ep,\ep}\\
   &F(x),&\, &x\in \Gamma_{\ep,\ep}\setminus \Om.
  \end{array}
  \right.
 \end{equation}
 In particular, $u_\ep|_{\partial \Om}=g$.
\end{theorem}

\begin{proof}
 Following the idea of the proof of Theorem~\ref{eps-exist-meas} we define a function $u_0:\Om_\ep\to\R$ as
 \[
 u_0(x):=c<\inf_{\Gamma_{\ep,\ep}}F(x),
 \]
  Moreover, for bounded functions $u:\Om_\ep\to \R$ we let $T$ be an operator defined as follows:
\[
 Tu(x):=\left(1-\frac{\dist(x,\Gamma_{\ep})}{\ep}\right)F(x)+\frac{\dist(x,\Gamma_{\ep})}{\ep}\vint_{B(x,\ep)} u(y)d\mu(y).
\]
By convention, we will interpret that
\begin{align*}
 &\left(1-\frac{\dist(x,\Gamma_{\ep})}{\ep}\right)F(x)=0\quad \hbox{and} \quad \frac{\dist(x,\Gamma_{\ep})}{\ep}=1
\end{align*}
for all $x\in \Om\setminus\Gamma_{\ep,\ep}$.

 By an iterative scheme we define $u_{i+1}:=Tu_{i}$ for $i=0,1,\ldots$ and show that $\{u_i\}_{i=0}^{\infty}$ is a monotone increasing bounded sequence of continuous functions in $\Om_\ep$. Furthermore, as in Lemma~\ref{eps-exist-meas}, the sequence $\{u_i\}_{i=0}^{\infty}$ is uniformly bounded from above by $\sup_{\partial \Om} g<\infty$ and the argument for this to hold follows again from definitions of $u_0$, $T$ and induction.

 The natural modification of the proof of Theorem~\ref{eps-exist-meas} allows us to conclude that function $u_\ep:=\lim_{i\to\infty}u_i$ is continuous and satisfies \eqref{eps-exist-cont-est}.

 In order to show the uniqueness, let us suppose that $u^1$ and $u^2$ solve \eqref{eps-exist-cont-est} and $u^1\not\equiv u^2$. Set
 \[
  \Om':=\{x\in \Om: u^1(x)-u^2(x)=\sup_{z\in \overline{\Om}} (u^1(z)-u^2(z))=M>0\}.
 \]
 Furthermore, note that it is enough to consider $\sup_{z\in \overline{\Om}} (u^1(z)-u^2(z))$ instead of the supremum for absolute values of $u^1-u^2$, as the case $u^2-u^1$ follows by the symmetric argument. The definition of $\Om'$ immediately implies that $\Om'$ is non-empty. Moreover, let us notice that $\Om'$ contains a ball $B(x, \ep)$. Indeed, let us consider two cases.

 {\sc Case 1.} $\Om'\setminus \Gamma_{\ep,\ep}\not=\emptyset$. Then for $x\in \Om'\setminus \Gamma_{\ep,\ep}$ it holds that
 \begin{equation}\label{eq11-exist-cont}
  0=u^1(x)-u^2(x)-M=\vint_{B(x,\ep)} \left(u^1(y)-u^2(y)-M\right) d\mu(y)\leq 0.
 \end{equation}
 Hence, $u^1-u^2\equiv M$ in $B(x, \ep)$. Choose a point $z\in \partial \Om$ and some point $x'\in \Om$ close enough to $z$. Since $\Om$ is path-connected, there exists a continuous curve $\gamma$ joining $x$ and $x'$ in $\Om$. By the compactness of $|\gamma|$ we may find a finite cover $\mathcal{C}=\{B_i\}_{i=1}^{N}$ of $\gamma$ by balls centered at points $x_i\in\gamma$ with radii $\ep/2$ such that $x_1=x$ and $x_N=x'$. Let $B(x'', \ep)$ be a ball in $\mathcal{C}$ with $x''\in \gamma$ and $x''\in B(x,\ep/2)\cap B(x_2,\ep/2)$. We apply reasoning at \eqref{eq11-exist-cont} to $B(x'', \ep)$ using again the mean value property for $u^1$ and $u^2$ and obtain that $u^1-u^2\equiv M$ in $B(x'', \ep)$. We continue this procedure along $\gamma$ till we reach first point $x'''\in\gamma$, such that $x'''\in \Om\setminus\Gamma_{\ep,\ep}\cap B(x_i, \ep)$ for some $2\leq i \leq N-1$. Then by the definition of $u_\ep$ in \eqref{eps-exist-cont-est} we get
 \begin{equation}\label{eq1111-exits-cont}
  0=u^1(x''')-u^2(x''')-M\leq\frac{\dist(x''',\Gamma_{\ep})}{\ep}\vint_{B(x''',\ep)} \left(u^1(y)-u^2(y)-M\right) d\mu(y)\leq 0.
 \end{equation}
   Thus, by repeating the last step at most once more, we have approached $\partial \Om$ obtaining a contradiction with the fact that $u^1|_{\partial \Om}=g=u^2|_{\partial \Om}$. Namely, for $z\in \partial \Om$ it holds that $u^1(z)-u^2(z)=0$, even though for points $y\in U\cap \Om$ for an arbitrarily small neighborhood $U$ of $z$ we have that $u^1(y)-u^2(y)=M>0$, contradicting continuity of $u^1$ and $u^2$.

  {\sc Case 2.} $\Om'\setminus \Gamma_{\ep,\ep}=\emptyset$. Then, since $\Om'$ is non-empty, there exists
   $x\in \Om\cap\Gamma_{\ep,\ep}$ and the above procedure simplifies. In fact we immediately reach the contradiction, since for $z\in \partial \Om\cap B(x,\ep)$ on one hand we have that $u^1(z)-u^2(z)=0$, but on the other hand for points $y\in U\cap \Om$ for an arbitrarily small neighborhood $U$ of $z$ it holds that $u^1(y)-u^2(y)=M>0$   by \eqref{eq1111-exits-cont} applied for $x''':=x$.

   The proof of the uniqueness and the whole proof of the theorem are, therefore, completed.
 \end{proof}

 We close this section with the following result showing that if we know that a Dirichlet problem has a continuous subharmonic solution, then the weakly harmonic solution exists and satisfies the same continuous boundary data.

\begin{theorem}\label{sub-Dir}
Let $(X,d,\mu)$ be a metric measure space satisfying the $\delta$-annular decay condition for some $\delta\in(0,1]$. Let $\Omega$ be a bounded domain in $X$ and consider a continuous function $g: \partial \Omega \rightarrow \mathbb{R}$. If there is a continuous subharmonic function $v \in S_-\harm(\Omega) \cap C(\overline{\Omega })$ such that $v|_{\partial \Om}=g$, then there exists a weakly harmonic function $u\in \harm(\Omega) \cap C(\overline{\Omega})$ such that $u|_{\partial \Omega} = g$.
\end{theorem}

\begin{proof}
Denote by $C_b( \Omega)$ a space of bounded continuous functions on $\Omega$. Let us define an operator $T: C_b(\Omega) \rightarrow C_b(\Omega)$ given by
\[
  Tu (x) = \dashint_{B(x,r_x)} u(y) d \mu (y)\ ,
 \]

where $r_x = \frac{1}{2}\textrm{dist}(x,\partial \Omega)$.

In order to see that $T$ is well defined, let us consider any $u \in C_b(\Omega)$ and denote $M:=\|u\|_{L^\infty(\Om)}$. The standard computations then imply that
\begin{align}
|Tu(x)-Tu(y)|&\leq \frac{1}{\mu(B(x,r_x))} \int_{B(x,r_x) \vartriangle B(y,r_y)}
|u(z)| d\mu(z) + \frac{|\mu(B(x,r_x))-\mu(B(y, r_y))|}{\mu(B(x,r_x)) \mu(B(y,
r_y))}\left|\int_{B(y,r_y)} u(z)d\mu(z)\right| \nonumber \\
 &\leq M\left(\frac{\mu(B(x,r_x) \vartriangle B(y,r_y))}{\mu(B(x,r_x))}+
\frac{\mu(B(x,r_x) \vartriangle B(y,r_x))}{\mu(B(x,r_x)) \mu(B(y, r_y))}\mu(B(y,
r_y))\right) \nonumber \\
&\leq 2M \frac{ \mu( B(x,r_x) \bigtriangleup B(y,r_y) )}{ \mu (B(x,r_x))}. \label{eq:szacowanie}
\end{align}

Moreover,
\[ B(x, \min \{r_y - d(x,y), r_x \} ) \subset B(x,r_x) \cap B(y,r_y)  \]
and
\[ B(x,r_x) \cup B(y,r_y)
\subset B(x,\max\left\{r_x ,r_y \right\} + d(x,y)). \]

Hence,
\begin{equation} \label{eq:pierscien}
B(x,r_x) \bigtriangleup B(y,r_y) \subset  B(x,\max\left\{r_x ,r_y \right\} + d(x,y)) \setminus B(x, \min \{r_y - d(x,y), r_x \} ).
\end{equation}

We combine (\ref{eq:szacowanie}) and (\ref{eq:pierscien}) and since, $r_y \rightarrow r_x$ as $y \rightarrow x$,
we get that the $\delta$-annular decay property of $X$ implies $Tu \in C( \Omega)$, cf. Definition~\ref{defn-an-dec}. By basic properties of the mean value one also obtains that $\|Tu\|_{\infty} \leq M$.

\bigskip

We use operator $T$ to construct the following sequence of functions:
\begin{itemize}
\item[(1)] $u_0:=v$ ,
\item[(2)] $u_n := Tu_{n-1}$ for $n=1,2,\ldots$.
\end{itemize}

We easily see that all $u_n \in C_b(\Omega)$ and $\|u_n \|_{L^\infty(\Om)} \leq \| v \|_{L^\infty(\Om)}$.

Next, observe that every $u_n$ is weakly subharmonic in $\Om$ with admissible radii $r^x = r_x$.
Indeed, $u_0 = v$ is subharmonic, so by the definition
\[
u_0(x)\leq \dashint_{B(x,r_x)} u_0 (y) d \mu (y) = T u_0 (x).
\]
Moreover, note that if $v \leq w$, then $Tv \leq Tw$. Hence, by induction we get
\[
u_n(x) \leq u_{n+1}(x) = T u_n(x) = \dashint_{B(x,r_x)} u_n (y) d \mu (y).
\]
The above reasoning shows also that sequence $\{u_n\}_{n=0}^{\infty}$ is increasing at every $x\in \Om$.

We extend sequence $\{u_n\}_{n=0}^{\infty}$ to $\overline{\Omega}$ in such a way that $u_n|_{\partial \Omega} = g$ for all $n=0,1,\ldots$. Indeed, this is possible for $u_0$ because $v$ is continuous up to the boundary and $v|_{\partial \Omega } = g$. In order to see that the same holds true for $u_n$ for $n\geq 1$ let us first consider any $w \in C(\overline{\Omega})$ with $w|_{\partial \Omega}=g$. Then for all $x \in \partial \Omega$ and $\epsilon > 0$ there exists $\delta > 0$ such that
\[
 \left| w(y) - g(x) \right| < \epsilon,
\]
for $y \in \Omega \cap B(x,\delta)$. Hence, for $y \in B(x,\frac{\delta}{2}) \cap \Omega$ we have
\[
 \left|Tw(y) - g(x) \right| \leq \dashint_{B(y,r_y)} \left| w(z) - g(x) \right| d \mu (z) < \epsilon.
\]
In a consequence, $Tw \in C(\overline{\Omega})$ and $Tw|_{\partial \Omega} = g$. We apply this reasoning with $w=u_n$ for $n=0,1,\ldots$ and obtain that $\{u_n\}_{n=1}^{\infty}$ is bounded increasing sequence of continuous functions on compact set, such that all $u_n|_{\partial \Omega} = g$.

\smallskip

We are now in a position to define the following Borel function:
\[
 u(x):= \lim_{n \rightarrow \infty } u_n(x)\quad\hbox{ for }x\in \Om.
\]
One can show that $\{u_n\}_{n=1}^{\infty}$ converges uniformly in $\Omega$ and the reasoning is similar to the one in the proof of Theorem 2.1 in \cite{lps}, cf. presentation in the proof of Theorem~\ref{eps-exist-meas} above. Since $u_n|_{ \partial \Omega } = g$ for all $n$, the sequence converges uniformly also in $\overline{\Omega}$. Hence, $u \in C(\overline{\Omega})$ and $u|_{ \partial \Omega } = g$. By employing the monotone convergence theorem to the sequence
\[
 u_{n+1}(x) = \dashint_{B(x,r_x)} u_n(y) d \mu (y)
 \]
we get that $u \in \wharm( \Omega )$.

\end{proof}

\subsection{The Perron method for harmonic functions}

In this section we employ the Perron method to show that for metric spaces with metrically continuous measures the Dirichlet problem in a domain $\Om$ with continuous boundary data has a strongly harmonic solution provided solvability of the appropriate Dirichlet problem on all balls whose closures are contained in $\Om$, see Theorem~\ref{glowne}. Furthermore, in Theorem~\ref{glowne-2} we prove that if, additionally at every $x\in \partial \Om$ there exists a barrier, then solvability on balls is also a necessary condition for solvability of the Dirichlet problem in $\Om$.

We begin with some auxiliary definitions and results. First, we generalize subharmonic functions, cf. Definition~\ref{defn-sub-super}. For the sake of the readers convenience we decided to present this notion here instead in Section~\ref{Sec-harm}. Such a generalization together with the above discussion of the Dirichlet problem will then be applied in the studies of Perron solutions.

\begin{definition}\label{def-loc-subh}
Let $f$ be a continuous function on $\Omega$. We say that $f$ is \emph{locally subharmonic in $\Omega$} if for every $x$ there exists $R_x>0$ such that $\overline{B(x,R_x)} \subset \Omega $ and for $0<\rho<R_x$
\begin{eqnarray}\label{sub}
f(x) \leq \frac{1}{ \mu (B(x,\rho)) } \int_{B(x,\rho)} f(z) d \mu(z).
\end{eqnarray}
\end{definition}

\begin{prop}\label{loc-subh}
 The strong and weak maximum principles hold for locally subharmonic functions in $\Om$.
\end{prop}
\begin{proof}
The proof of the first part of the proposition is similar to the one for the strong maximum principle for harmonic functions, see Theorem 3.1 in Gaczkowski-G\'orka~\cite{GG} and
Proposition~\ref{strong-max-pr}. Therefore, here we only provide a sketch of the reasoning.

Let $\Om'=\{x\in \Om: f(x)=\sup_{\Om}f\}$. Set $\Om'$ is closed. To see that it is open, let $x_0\in \Om'$ and $B(x_0, r)\subset \Om$ for $r<R_{x_0}$ as in Definition~\ref{def-loc-subh}. Then by the definition of locally subharmonic functions we have that
\[
 \int_{B(x_0, r)} (M-f(z)) d \mu(z)\leq \mu(B(x_0, r))(M-f(x_0))=0.
\]
Since $f\leq M$ in $\Om$, we have that $f\equiv M$ on $B(x_0, r)$ and $\Om'$ is open, hence $\Om'=\Om$.
As for the proof of the weak maximum principle, it follows from the strong one in the same way as in the case of harmonic functions, cf. the proof of Corollary~\ref{weak-max-pr}.
\end{proof}
\begin{prop}\label{prop-max-subh}
Maximum of two locally subharmonic functions in a domain $\Om$ is locally subharmonic in $\Om$.
\end{prop}
\begin{proof}
 Let $x\in \Om$ and suppose that $\max \{f(x), g(x)\}=f(x)$. Then
\[
  \max \{f(x), g(x)\}=f(x) \leq \frac{1}{\mu (B(x,\rho))} \int_{B(x,\rho)} f(z) d\mu(z) \leq \frac{1}{\mu (B(x,\rho)) } \int_{B(x,\rho)} \max\{f(z),g(z)\} d\mu(z),
\]
provided that $\rho <R = \min \{R_{f,x}, R_{g,x}\}$, where $R_{f,x}$ and $R_{g,x}$ are radii $R_x$ from Definition~\ref{def-loc-subh} for functions $f$ and $g$, respectively.
\end{proof}

Let $\Om \subset X$ and $g:\partial \Om\to \R$ be a measurable function. In what follows by $H_{\Om}[g]$ we denote a solution to the Dirichlet problem in $\harm(\Om, \mu)$ with a boundary data $g$, provided that it exists.

\begin{definition}\label{lower-Perron}
Let $\Om\subsetneq X$ be a domain and $g:\partial \Omega\to \R$ be a continuous function. We denote by $S_{g}$ the family of all locally subharmonic functions in $\Om$ such that $f \leq g$ on $\partial \Omega$. Define a \emph{lower Perron solution $P_{\Om}[g]$ in $\Om$} of $g$ as follows:
\[
P_{\Om}[g](x) := \sup_{f \in S_{g}} f(x),\quad x\in \Om.
\]
\end{definition}
If the underlying domain $\Om$ is fixed or clear from the context of discussion, then for the sake of simplicity we will write $P_{\Om}[g]=P[g]$. Note that under our assumption that $X$ is proper, we have that $\partial \Om$ is compact and balls are relatively compact. However, in what follows we need to assume additionally that balls are connected. This assumption will be used in the proof of Theorem~\ref{glowne}.

\begin{remark}\label{lowerP-bdd}
 Note that by the weak maximum principle in Corollary~\ref{weak-max-pr} it holds that
 \[
  P_{\Om}[g]\leq \sup_{x\in \partial \Om}g(x).
 \]
\end{remark}
\begin{definition}\label{def-harm-mod}
Let $\Om \subset X$ be an open set and $f:\Omega \rightarrow \mathbb{R}$ be a continuous function. Suppose that $x_0 \in \Omega$, $r>0$ are such that $\overline{B(x_0,r)} \subset \Omega$. Denote by $g=f|_{\partial B(x_0,r)}$. The following function is called \emph{a harmonic modification of $f$ in $B(x_0,r)$}:
\begin{equation}
\bar{f}_{x_0,r}(x)=
\begin{cases}
&f(x),\quad x\in \Omega \setminus B(x_0,r)\\
&H_{B(x_0,r)}[g], \quad x\in B(x_0,r).
\end{cases}
\end{equation}
\end{definition}

The harmonic modification appears in the literature also as \emph{the Poisson modification} and has several variants e.g. for superminimizers and superharmonic functions in the setting of Newtonian spaces, see e.g. Sections 8.7 and 10.9 in Bj\"orn--Bj\"orn~\cite{bb}.

The following observations easily follow from the weak maximum principle for locally subharmonic functions (Proposition~\ref{loc-subh}) and the definition of the harmonic modification.
\begin{lemma}\label{uwagamod}
Let $f$ be a locally subharmonic function in $\Om$, then  $\bar{f}_{x,r} \geq f$ for any $x\in \Om$ and $\overline{B(x, r)}\subset \Om$.
\end{lemma}

\begin{lemma}\label{Lemmonotone}
Let $g:\partial \Omega \to \R$ be a continuous function and $f,h \in S_g$ be such that $f \geq h$ in $\Om$. Then $\bar{f}_{x, r} \geq \bar{h}_{x, r}$ for any $x\in \Om$ and $\overline{B(x, r)}\subset \Om$.
\end{lemma}

\begin{prop}\label{prop-sub-hmod}
Let $f$ be a locally subharmonic function in $\Om$, then harmonic modification $\bar{f}_{x,r}$ is locally subharmonic in $\Om$ for any $x\in \Om$ and $\overline{B(x, r)}\subset \Om$.
\end{prop}
\begin{proof}
Let $\bar{f}_{x,r}$ be a harmonic modification of $f$ in $B(x, r)$. If $y \in B(x,r)$, then by harmonicity of $\bar{f}_{x,r}$ in $B(x,r)$, we have that inequality (\ref{sub}) holds for $\rho < R_y$, where $R_y$ is small enough.
Otherwise, if $y \notin B(x,r)$, then Lemma \ref{uwagamod} implies for sufficiently small $\rho>0$, that
\[
\bar{f}_{x,r}(y) = f(y) \leq \frac{1}{ \mu (B(y,\rho)) } \int_{B(y,\rho)} f d \mu  \leq \frac{1}{ \mu (B(y,\rho)) } \int_{B(y,\rho)} \bar{f}_{x,r} d \mu.
\]
This completes the proof.
\end{proof}

\begin{definition}\label{def-barr}
Let $x_0 \in \partial \Omega$. We say that a continuous function $f:\Om\to\R$ is a \emph{barrier for $\Omega$ at $x_0$} if:
\begin{itemize}
\item[(1)] $f$ is locally subharmonic in $\Om$,
\item[(2)] $f(x_0)=0$,
\item[(3)] $f(x)<0$ for every $x\in \partial \Omega \setminus \{x_0\}$.
\end{itemize}
\end{definition}

The notion of a barrier plays an important role in the potential analysis, see e.g. Chapter 11 in Bj\"orn--Bj\"orn~\cite{bb}. We point out that, upon slight modification of Definition~\ref{def-barr}, one can show that existence of a barrier at a given boundary point $x_0$ of a domain is equivalent to the fact that $x_0$ is a regular point, see Theorems 11.2 and 11.11 in \cite{bb}. Recall, that a point $x_0$ is regular if
\[
 \lim_{\Om\ni y\to x}H_{\Om}g(y)=g(x),
\]
for all continuous boundary data $g:\partial \Om\to \R$. From that perspective the following result provides, among other results, conditions on the boundary enabling us to conclude that all boundary points are regular.
\begin{theorem}\label{glowne}
Let $(X, d, \mu)$ be a metric space with metrically continuous measure $\mu$ and such that balls are connected. Let $\Omega \subset X$ be a domain and $g:\partial \Om \to \R$ be a continuous function. If for every ball $B$ in $\Om$ such that $\overline{B}\subset\Om$ there exists a solution of the Dirichlet problem with a boundary data $g|_{\partial B}$, then $P_{\Om}[g]$ is a harmonic function on $\Omega$. Moreover, if  $\Omega$ has a barrier at $x_0 \in \partial \Omega$, then
\begin{eqnarray*}
\lim_{\Om \ni y \rightarrow x_0} P_{\Om}[g](y) = g(x_0).
\end{eqnarray*}
\end{theorem}

\begin{proof}
By the definition of harmonic functions (see Definition~\ref{defn-harm}) it is enough to show that
\[
P_{\Om}[g]:=P[g] \in \mathcal{H}(B(x,r)),
\]
for all balls $B(x,r)$ satisfying $\overline{B(x,r)} \subset \Omega$. Let us choose any of such balls and fix it.

In order to show the harmonicity of $P[g]$ in $B(x,r)$ let us take a sequence $ \{u_{k}\}_{k=1}^{\infty} \subset S_{g}$, such that
\begin{eqnarray*}
\lim_{k \rightarrow \infty } u_{k}(x) = P[g](x).
\end{eqnarray*}
Next, we define
\begin{eqnarray}
v_{k} = \max \{ u_{1}, \ldots , u_{k} \}.
\end{eqnarray}
By Proposition~\ref{prop-max-subh} it holds that $v_{k}\in S_g$ for every $k$. Let $w_k$ be a harmonic modification of $v_k$ on $B(x,r)$, i.e. $w_k=\overline{v_k}_{x,r}$. From Definition~\ref{def-harm-mod} and Proposition~\ref{prop-sub-hmod} we infer that $w_k \in S_g$.

Since $\{v_k\}_{k=1}^{\infty}$ is a monotone increasing sequence, we get by Lemma \ref{Lemmonotone} that $\{w_k\}_{k=1}^{\infty}$ is monotone increasing as well. Moreover, by Remark~\ref{lowerP-bdd} we have that
\begin{eqnarray}\label{ciagi}
  u_k(y) \leq v_k(y) \leq w_k(y) \leq P[g](y)\leq\sup_{\partial \Om} g<\infty.
\end{eqnarray}
Thus, the pointwise limit $\displaystyle w(y) = \lim_{k \rightarrow \infty} w_k(y)$ exists and from (\ref{ciagi}) we get that $w(x) = P[g](x)$ and $w \leq P[g]$ in $\Om$. 

Furthermore, since $w_k$ is a monotone sequence of harmonic functions in $B(x,r)$, we get by the Lebesgue Convergence Theorem that
\[
	w(y)= \frac{1}{ \mu (B(y,\delta)) } \int_{B(y,\delta)} w d\mu
\]
for any $B(y, \delta)\Subset B(x,r)$. This together with metric continuity of measure $\mu$ imply that $w$ is continuous in $B(x,r)$, cf. Proposition~\ref{harm-cont}. Hence, we get that $w \in \mathcal{H}(B(x,r))$.

In order to complete the proof we need to show that $P[g] \leq w$ in $B(x,r)$. For this purpose let us consider
$u \in S_g$. We have that
\[
 \lim_{k\to \infty}\max \{w_k,u\}(x)=w(x)\quad\hbox{and}\quad\max \{w_k,u\} \in S_g.
\]
Denote by $h_k$ a harmonic modification of $\max \{w_k,u\}$ in $B(x,r)$. Then,
\[
 \sup_{z \in \partial \Omega} g(z) \geq P[g](y)\geq h_k(y) \geq \max \{w_k,u\}(y).
\]
Next, taking $y=x$, we have
\[
 w(x)=P[g](x)\geq h_k(x) \geq \max \{w_k,u\}(x),
\]
and we get that $\lim_{k \rightarrow \infty} h_k(x) =w(x)$. Thus, by harmonicity of $w$ and all $h_k$ we have
\begin{eqnarray*}
  \lim_{k \rightarrow \infty } \frac{1}{\mu (B(x,r)) } \int_{B(x,r)} h_{k} d \mu = \frac{1}{ \mu (B(x,r)) } \int_{B(x,r)} w d \mu.
\end{eqnarray*}
Subsequently, from the Lebesgue Convergence Theorem we get
\begin{align*}
\frac{1}{ \mu (B(x,r)) } \int_{B(x,r)} w d \mu &= \lim_{k \rightarrow \infty } \frac{1}{ \mu (B(x,r)) } \int_{B(x,r)} h_{k} d \mu \\
&\geq \lim_{k \rightarrow \infty } \frac{1}{ \mu (B(x,r)) } \int_{B(x,r)} \max \{ w_{k},u \} d \mu \\
&=  \frac{1}{ \mu (B(x,r)) } \int_{B(x,r)} \max \{ w,u \} d \mu.
\end{align*}
In a consequence we obtain that $w \geq u$ on $B(x,r)$ for any $u\in S_g$. Hence, $P[g] \leq w$ in $B(x,r)$. This ends the proof of harmonicity of $P[g]$.

Now, we shall show the second part of the theorem. Let $f$ be a barrier for $\Omega$ at $x_0$.
Fix $\epsilon > 0$, since $g$ is continuous on $\partial \Omega$, there exists $r>0$ such that
\begin{eqnarray*}
g(x_0) - \epsilon < g(y) < g(x_0) + \epsilon
\end{eqnarray*}
holds for all $y \in \partial \Omega \cap B(x_0,r)$.
Hence, for any $C>0$ the inequality
\begin{eqnarray} \label{nierownosc}
g(x_0) - \epsilon +Cf(y) < g(y) < g(x_0) + \epsilon - Cf(y)
\end{eqnarray}
holds in $\partial \Omega \cap B(x_0,r)$. Moreover, since $f$ is negative on $\partial \Omega \setminus B(x_0,r)$,
there exists $C$, such that inequality (\ref{nierownosc}) holds on $\partial \Omega$.
Thus, it holds that $g(x_0) - \epsilon +Cf \in S_{g}$, and by the definition of the lower Perron solution
\begin{eqnarray}\label{jeden}
g(x_0) - \epsilon +Cf \leq P[g]\quad\hbox{in }\Om.
\end{eqnarray}

Let now $h \in S_{g}$. We have that
\begin{eqnarray*}
h \leq g < g(x_0) + \epsilon - Cf \quad\hbox{on }\partial \Omega.
\end{eqnarray*}
Therefore, $h + Cf < g(x_0) + \epsilon$ on $\partial \Omega$ and by the weak maximum principle we get that
\begin{eqnarray*}
h + Cf < g(x_0) + \epsilon\quad\hbox{on }\bar{\Omega}.
\end{eqnarray*}
Hence, for all $y\in\Om$
\begin{eqnarray}\label{dwa}
P[g](y) \leq g(x_0) + \epsilon - Cf(y).
\end{eqnarray}
Finally, (\ref{jeden}) and (\ref{dwa}) result in $P[g](y)=g(x_0) + \epsilon - Cf(y)$ for all $y\in \Om$. The continuity of $f$ together with the Property (2) of barrier functions (cf. Definition~\ref{def-barr}) imply the second assertion of the theorem, as $\epsilon>0$ is chosen arbitrarily small. The proof of Theorem~\ref{glowne} is, thus, completed.
\end{proof}

The following result is one of main theorems of the paper. We establish the equivalence between the solvability of the Dirichlet problem in the underlying domain and the solvability of Dirichlet problems in all balls contained in the domain.

\begin{theorem}\label{glowne-2}
Let $(X, d, \mu)$ be a metric space with metrically continuous measure such that all balls are connected. Suppose that $\Omega \subset X$ is a domain and $g:\partial \Om \to \R$ is a continuous function. Moreover, let us assume that the harmonic Dirichlet problem has a solution on an arbitrary ball $B\subset \Om$ with a boundary data $g|_{\partial B}$.

Then, the Dirichlet problem has a solution in $\Om$ if and only if at every $x \in \partial \Omega$ there exists a barrier function, and for every ball in $\Omega$ there exists a solution of the Dirichlet problem with $g$.
\end{theorem}
\begin{proof}
The necessity part of the proof of Theorem~\ref{glowne-2} immediately follows from Theorem \ref{glowne}.

In order to show the sufficiency part, let us assume that the Dirichlet problem has a solution in $\Om$ for a boundary data $g$ as in assumptions of the theorem. We fix any $x \in \partial \Omega$ and consider $y \in \partial \Omega$, a point distinct from $x$. By the Urysohn lemma, there exists a continuous function $h: \partial \Omega \rightarrow [-1,0]$ such that $h(x)=0, h(y)=-1$ and $h|_{\partial \Omega \setminus \{x, y\}}\in (-1,0)$. Let $G$ be a unique solution to the Dirichlet problem in $\Om$ with boundary condition $g:=h$. Then $G$ is a barrier in $\Om$.
\end{proof}

\section{The Liouville theorem}

The Liouville theorem is a classical result in the theory of harmonic functions in $\R^n$. The purpose of this section is to establish similar results for strongly and weakly harmonic functions on metric measure spaces. It turns out that already on $\R$ we may choose such a measure, so that the Liouville theorem fails, cf. Example~\ref{ex4}. However, below we establish a fairly general condition on a measure resulting in the Liouville theorem, see \eqref{theoremLP-assp} and \eqref{theoremLP-assw} in Theorem~\ref{theoremLP}. Moreover, we discuss some sufficient conditions on a measure and a metric space to ensure that Theorem~\ref{theoremLP} holds, see Remark~\ref{rem-suff-liouv} and Corollary~\ref{cor-Liouv}. For strongly harmonic functions a variant of the Liouville theorem follows from the Harnack inequality on balls, see Theorem~\ref{Liouv-Harn}.

Let us begin with the following simple observation.
\begin{prop}\label{Liouv-prop1}
Suppose that $f \in \mathcal{H}(\Om, \mu )$ and let $f > 0$ in $\Om$,
then $g \in \mathcal{H}(\Om, f \mu)$ if and only if $ gf \in \mathcal{H}(\Om, \mu)$.

Moreover, the assertion remains true for $f, g\in \wharm(\Om, \mu)$ and $f>0$ provided that $f$ and $g$ have the same sets of admissible radii $r^x_i$ for $i=1,2,\ldots$ at every point $x\in\Om$.
\end{prop}
\begin{proof}
Let $g \in \mathcal{H}(\Om, f\mu )$. Define $h=fg $ and by a straightforward computation we get
\begin{eqnarray*}
\frac{\int_{B(x,r)} h d \mu}{\mu ( B(x,r) )} =\frac{\int_{B(x,r)} gf d \mu}{\mu ( B(x,r) )} = \frac{g(x) {\int_{B(x,r)} f d \mu }}{ \mu ( B(x,r) )} = f(x)g(x) = h(x).
\end{eqnarray*}
Hence, we obtain that $h \in \mathcal{H}(\Om, \mu)$.

In order to show the opposite implication, let us define $g := \frac{h}{f}$, where $h \in \mathcal{H}(\Om, \mu )$. Thus,
\begin{eqnarray*}
\frac{\int_{B(x,r)} \frac{h}{f}f d \mu}{\int_{B(x,r)} f d \mu } = \frac{\int_{B(x,r)} h d \mu}{\int_{B(x,r)} f d \mu } = \frac{h(x)}{f(x)}.
\end{eqnarray*}

Let now $f,g\in \wharm(\Om, \mu)$ and $f>0$. Then, the above reasoning holds at every ball $B=B(x, r^x_i)$ for all $i$, since sets of admissible radii are the same for both functions.
\end{proof}

Before proving Liouville-type results, let us give an example illustrating that, in general, the Liouville property need not hold.
\begin{example}\label{ex4}
 There exists a measure $\mu$ on $(\mathbb{R}, |\cdot |)$ and $f \in \mathcal{H}(\mathbb{R}, \mu)$ which is bounded and nonconstant.

 Let $f(x)= 2 e^{x}\cosh x=1+e^{2x}$. Then $f \in \mathcal{H}(\mathbb{R}, e^{-x} dx )$. Indeed, for every $x \in \mathbb{R}$ and $r>0$ we have
\begin{equation*}
\int_{B(x,r)} f(t)e^{-t} dt\,=\, \int_{x-r}^{x+r} (1+e^{2t})e^{-t} dt.
\end{equation*}
Thus,
\begin{equation}\label{ex-Liouv-eq1}
\int_{x-r}^{x+r} e^{-t} dt = e^{-x}(e^r - e^{-r})
\end{equation}
and
\begin{equation*}
\int_{x-r}^{x+r} e^{2t} e^{-t} dt = e^{x}(e^r - e^{-r}).
\end{equation*}
Hence, we obtain that
\begin{equation*}
 \frac{1}{\muballxr} \int_{B(x,r)} f(t)\mu(t)\,=\,1+\frac{\int_{x-r}^{x+r} e^{2t} e^{-t} dt }{\int_{x-r}^{x+r} e^{-t} dt } = 1+e^{2x}=f(x).
\end{equation*}

However, by virtue of Proposition~\ref{Liouv-prop1} applied with $f(x)= 2 e^{x}\cosh x$, measure $\mu=e^{-x}$ we get that a entire nonconstant bounded harmonic function $g:=\frac{1}{f}=\frac{1}{1 + e^{2x}} \in \mathcal{H}(\mathbb{R}, 2 \cosh x \, dx)$, since $gf=1\in \harm(\Om ,\mu)$ by \eqref{ex-Liouv-eq1}.
\end{example}

We turn now to the question of the structure and dimension of the space of harmonic functions on the whole space. Similar studies in the setting of manifolds have been studied by several authors, e.g. Colding--Minicozzi~\cite{cm}.

Definition~\ref{defn-harm} implies that the space of harmonic functions $\harm(\Om, \mu)$ is a linear space. The following result holds.

\begin{theorem} \label{tw2}
  Denote by $\dim \harm (\Om, \mu)$ a dimension of $\harm (\Om, \mu)$ as a linear space. Then $\dim \mathcal{H} (\mathbb{R}, d,  \mu) \leq 2$.
\end{theorem}

Note that the similar result for weakly harmonic functions fails since such functions do not have a natural structure of a linear space. In order to ensure such structure one would have to assume, for instance, that all functions in $\wharm(\Om, \mu)$ have the same sets of admissible radii at every point of $\Om$.

Before proving Theorem~\ref{tw2} we provide an example showing that in fact $\dim \mathcal{H}$ can be smaller then $2$.

\begin{example}\label{ex-Liouv2}
It holds that $\dim \mathcal{H} (\mathbb{R}, |\cdot|, |x| dx)=1$.

Indeed, let $f \in \mathcal{H}( \mathbb{R}, |\cdot|, |x| dx)$ and denote $g:=f|_{ (0,+\infty)}$. Then  $g\in \mathcal{H}((0, + \infty), |\cdot|, xdx)$. Since $ x \in \mathcal{H}( (0, + \infty), |\cdot|, dx)$, then Proposition~\ref{Liouv-prop1} implies that  $xg\in \mathcal{H}( (0, 1), |\cdot|, dx)$ and, therefore, $g(x) = \frac{A}{x} + B$ for some positive constants $A, B$. The continuity of $f$ results in $f \equiv B$.
\end{example}

\begin{proof}[Proof of Theorem~\ref{tw2}]
We first show the following claim: If $f \in \mathcal{H} (\mathbb{R} , d, \mu) $, then $f$ is constant or strictly monotone.

Let $D= \{(x,y): x<y \}$ and consider a function $g: D \rightarrow \mathbb{R}$ defined as follows
\begin{equation*}
g(x,y):= f(x) - f(y).
\end{equation*}
Since $f$ is continuous in $D$ by Proposition~\ref{harm-cont}, then so is $g$ in $D\times D$. Suppose that $f$ is not strictly monotone. Then $g(x,y)=0$ for some $(x,y)\in D$, then $f(x) = f(y)$ and by the weak maximum principle in Corollary~\ref{weak-max-pr} we get that $f$ is constant on the interval $[x,y]$.

Next, we take any $b>y$. If $f$ is not strictly monotone on $[y,b]$, then we split this interval into the intervals where $f$ is monotone and apply the following reasoning on such intervals. Therefore, let us suppose that $f$ is monotone increasing on $[y,b]$, then by the the strong maximum principle in Proposition~\ref{strong-max-pr} applied to $[x,b]$, we get that $f$ is constant on $[x,b]$. We obtain the same conclusion if $f$ is monotone decreasing, since then we use the strong minimum principle (an immediate consequence of Proposition~\ref{strong-max-pr}). The analogous reasoning gives us that $f$ must be constant on any interval $[a,y]$ for $a<x$. From this, we have that $f$ is constant on any interval $[a,b]$ containing the set $[x,y]$ and the claim is proven.

Now, we are in position to prove the assertion of the theorem. Let $f,g \in \mathcal{H} (\mathbb{R}, d, \mu)$ be such that $f, g$ are non-constant. Then, by the claim $f$ and $g$ are strictly monotone. Hence, there exists $A \in \mathbb{R}$ such that
\begin{eqnarray*}
f(1) - f(-1) = A( g(1) - g(-1)).
\end{eqnarray*}
Thus,
\begin{eqnarray*}
f(-1) = A g(-1) + B, \\
f(1) = A g(1) + B
\end{eqnarray*}
for some $B \in \mathbb{R}$. Hence, by the maximum principle we get $f(x) - Ag(x) - B = 0$ for $x \in [-1,1]$. This implies that $f(x) - Ag(x) - B = 0$ for $ x \in \mathbb{R}$ and the proof of theorem follows.
\end{proof}

\begin{theorem} \label{theoremLP}
Suppose that for every $x, y \in X$ the following condition holds
\begin{eqnarray}\label{theoremLP-assp}
\liminf_{ r \rightarrow \infty } \frac{ \mu ( B(x,r) \bigtriangleup B(y,r)) }{  \mu (B(x,r) )} = 0.
\end{eqnarray}
Then, every bounded harmonic function in $\harm(X, \mu)$ is constant.

Moreover, the assertion holds true for a bounded $f \in \wharm(X, \mu)$ provided that at every $x, y\in X$
there exist sequences $(r^x_n), (r^y_n)$ of admissible radii such that the following holds
\begin{equation} \label{theoremLP-assw}
 \lim_{ n \rightarrow \infty } \frac{ \mu ( B(x,r^x_n) \bigtriangleup
B(y,r^y_n)) }{  \mu (B(x,r^x_n) )} = 0.
\end{equation}

\end{theorem}

Before proving the theorem we present two observations regarding sufficient conditions for functions and for a space and a measure for \eqref{theoremLP-assw} and \eqref{theoremLP-assp} to hold, respectively.

\begin{remark}
Suppose that at every point $x\in X$: (1) $f$ has the same sets of admissible radii $r^x_n$ for $n=1,2,\ldots$,\,and
 (2) $r^x_M=\sup_{n\in \N}r^x_n=\infty$. Then in assumption \eqref{theoremLP-assw} one can consider, for instance, sequences $r^x_n=r^y_n$ for $n=1,2,\ldots$.
\end{remark}

\begin{remark}\label{rem-suff-liouv}
 Let us provide an example of a measure which ensures that \eqref{theoremLP-assp} holds. Suppose that $(X, d, \mu)$ is a length metric measure space with a doubling measure $\mu$. For such spaces Corollary 2.2 in Buckley~\cite{buc} stays that $X$ satisfies a $\delta$-annular decay property for some $\delta\in(0,1]$, cf. the discussion following Definition~\ref{defn-an-dec}. Then $\mu$ satisfies \eqref{theoremLP-assp}. Indeed, as in Remark~\ref{remark-ann-dec} we have that for $x,y\in X$ with $d(x,y)<r$ it holds
 \begin{equation*}
\frac{\mu(B(x,r) \bigtriangleup B(y,r))}{\mu(B(x,r))} \leq A \left(\frac{2d(x,y)}{r+d(x,y)}\right)^\delta \frac{\mu(B(x,r+d(x,y))}{\mu(B(x,r))}\leq A C_{\mu}\left(\frac{2d(x,y)}{r+d(x,y)}\right)^\delta,
 \end{equation*}
 since $r+d(x,y)<2r$ and the doubling condition can be applied. By letting $r\to\infty$ we arrive at \eqref{theoremLP-assp}.
\end{remark}

\begin{proof}[Proof of Theorem~\ref{theoremLP}]
Let $f\in \mathcal{H}(X)$ be bounded and set $M:= \|f\|_{L^{\infty}(X)}$. We follow the steps of the reasoning at
\eqref{harm-cont-est}, see the proof of Proposition~\ref{harm-cont} and cf. Lemma 4.3 in \cite{GG}, and obtain
\begin{align*}
|f(x)-f(y)| &\leq \frac{1}{\mu(B(x,r_1))} \int_{B(x,r_1) \vartriangle B(y,r_2)}
|f(z)| d\mu(z)
  + \frac{|\mu(B(x,r_1))-\mu(B(y, r_2))|}{\mu(B(x.r_1)) \mu(B(y,
r_2))}\left|\int_{B(y,r_2)} f(z)d\mu(z)\right| \\
 &\leq M\left(\frac{\mu(B(x,r_1) \vartriangle B(y,r_2))}{\mu(B(x,r_1))}+
\frac{\mu(B(x,r_1) \vartriangle B(y,r_1))}{\mu(B(x,r_1)) \mu(B(y, r_2))}\mu(B(y,
r_2))\right) \\
&\leq 2M \frac{ \mu( B(x,r_1) \bigtriangleup B(y,r_2) )}{ \mu (B(x,r_1))}.
\end{align*}

Now let $r_1 = r_2 = r$. We take $\liminf$ for $r\rightarrow \infty$ and thus, by assumption
\eqref{theoremLP-assp}, we get that $f(x)=f(y)$ for every $x, y \in X$. From this the proof for strongly harmonic functions follows.

Let now $f\in \wharm(X, \mu)$. Then we set $r_1=r^x_n$ and $r_2 = r^y_n$ for $n=1,2,\ldots$ and appeal to \eqref{theoremLP-assw} in order to complete the proof of the theorem for weakly harmonic functions.

\end{proof}

\begin{corol}\label{cor-Liouv}
If $ \mu(X) < \infty$, then every bounded $f \in \mathcal{H}(X)$ is constant. Moreover, if $f\in \wharm (X)$ is bounded and $r^x_M = \infty$ for all $x \in X$, then $f$ is constant.
\end{corol}

\begin{proof}
Observe that for $x,y \in X$ and $r > d(x,y)$ we have
\begin{equation*}
B(x,r - d(x,y)) \subset B(x,r) \cap B(y,r) \subset B(x,r) \cup B(y,r) \subset B(x,r + d(x,y)).
\end{equation*}
Hence,
\begin{equation}\label{cor-Liouv-e1}
B(x,r) \bigtriangleup B(y,r) \subset B(x,r+d(x,y))\setminus B(x,r-d(x,y)).
\end{equation}
Next, for $n\geq 2$ we define $r_n = (2n - 1) d(x,y)$ and $\bar{r}_n = 2n d(x,y)$. Since
\begin{eqnarray*}
X = B(x,r_2) \cup \bigcup_{n =2}^\infty B(x, r_{n+1})\setminus B(x, r_n)
\end{eqnarray*}
and $ \mu(X) < \infty $, we have
\begin{eqnarray*}
\lim_{n \rightarrow \infty} \mu \left( B(x, r_{n+1})\setminus B(x, r_n)\right)= 0.
\end{eqnarray*}
In view of the above relations, we conclude
\begin{equation*}
\frac{ \mu( B(x,\bar{r}_n) \bigtriangleup B(y,\bar{r}_n) )}{ \mu (B(x,\bar{r}_n))} \leq \frac{\mu \left( B(x, r_{n+1})\setminus B(x, r_n)\right)}{ \mu (B(x,\bar{r}_2))} \rightarrow 0, \quad \hbox{ for } n\to\infty.
\end{equation*}
Therefore, assumption \eqref{theoremLP-assp} of Theorem~\ref{theoremLP} is satisfied and hence every bounded harmonic function in $\harm(X, \mu)$ is constant. The proof of the theorem for strongly harmonic functions is completed.

Now let $f \in \wharm(X)$ with $r^x_M=\infty$ for all $x \in X$. Therefore, at every $x\in X$ we may choose monotone sequences of admissible radii $(r^x_n)$ and $(r^y_n)$, such that $\displaystyle \lim_{n \rightarrow \infty} r^x_n =  \lim_{n \rightarrow \infty} r^y_n = \infty$.   Then,
\begin{equation*}
B(x, \min \{r^y_n - d(x,y), r^x_n \} ) \subset B(x,r^x_n) \cap B(y,r^y_n) \subset B(x,r^x_n) \cup B(y,r^y_n)
\subset B(x,r^x_n + r^y_n + d(x,y)).
\end{equation*}
Define two sequences
\[
r_n = r^x_n + r^y_n + d(x,y), \qquad s_n =  \min \{r^y_n - d(x,y), r^x_n \}\quad \hbox{for } n=1,2,\ldots.
\]
Hence,
\begin{equation*}
B(x,r^x_n) \bigtriangleup B(y,r^y_n) \subset  B(x,r_n) \setminus B(x, s_n )\quad  \hbox{ for all }n.
\end{equation*}
Let us construct the following subsequences of $(r_n)$ and $(s_{n})$:
\begin{itemize}
\item[(1)] $r_1':=r_{n_1}$, $s_1':=s_{n_1}$, for some $n_1\geq 1$,
\item[(2)] for $l=2,3,\ldots$ we set $s_l':=s_{n_{k+1}}$, such that $s_{n_{k+1}}> r_{n_k}$ and $r_l':=r_{n_{k+1}}$.
\end{itemize}
Therefore, for $i \neq j$ we have that
\[  \left( B(x,r_{i}') \setminus B(x, s_i' ) \right) \cap  \left( B(x,r_j') \setminus B(x, s_j' ) \right) = \emptyset.
\]
Hence, by additionally appealing to the finiteness of the measure $\mu$ of $X$, we get
\begin{equation*}
\frac{ \mu( B(x,r^x_{n_k}) \bigtriangleup B(y,r^y_{n_k}) )}{ \mu
(B(x,r^x_{n_k}))} \leq \frac{\mu \left( B(x, r_n')\setminus B(x,
s_n')\right)}{ \mu (B(x,r^x_1))} \rightarrow 0, \quad \hbox{ for }
n\to\infty.
\end{equation*}
Thus, assumption \eqref{theoremLP-assw} of Theorem~\ref{theoremLP} is satisfied implying that $f$ is constant. The proof of the theorem is, therefore, completed.
\end{proof}

The Liouville theorem can also be obtain from the Harnack inequality on balls, see Lemma~\ref{Harnack-on-balls}.
Below we assume that $\mu$ is bounded, restricting the set of admissible measures, but on the other hand we require harmonic function to be bounded from below only. Namely, the following result holds.
\begin{theorem}\label{Liouv-Harn}
 Let $X$ be a geodesic metric measure space with doubling measure $\mu$. Then, every bounded from below harmonic function in $\harm(X, \mu)$ is constant.
\end{theorem}
\begin{proof}
 Let $f\in \harm(X, \mu)$ and define $g=f-\inf_{X} f\geq 0$. By Proposition~\ref{further-prop}(\ref{further-prop-1}) we have that $g\in\harm(X, \mu)$. By the Harnack inequality, see Lemma~\ref{Harnack-on-balls}, we have that for all $x\in X$ and any ball $B(x,r)\in X$
 \[
  g(x)\leq \sup_{B(x,r)} g \leq C_{\mu}^3 \inf_{B(x,r)} g \to 0,\quad \hbox{ as } r\to \infty.
\]
Hence, $g\equiv 0$ and, in turn, $f$ is constant.
\end{proof}

In the setting of weakly harmonic functions the same type of argument cannot be applied. Indeed, if $r^B_M\to \infty$, then the Harnack constant $C_H$ in Lemma~\ref{Harnack-on-balls} grows unbounded.


\begin{thebibliography}{999}

\bibitem{bb} {\sc A. Bj\"orn, J. Bj\"orn}, \emph{Nonlinear Potential Theory on Metric Spaces}, EMS Tracts in Mathematics, 17, European Math. Soc., Zurich, 2011.

\bibitem{bss} {\sc K. Bogdan, A. St\'os, P. Sztonyk}, \emph{Potential theory for L\'evy stable processes}, Bull. Polish Acad. Sci. Math. 50(3) (2002), 361-–372.

\bibitem{buc} {\sc S. Buckley}, \emph{Is the maximal function of a Lipschitz function continuous?}, Ann. Acad. Sci.
    Fenn. Math. 24(2) (1999), 519-–528.

\bibitem{cheeg} {\sc J. Cheeger}, \emph{Differentiability of Lipschitz functions on metric measure spaces}, Geom. Funct. Anal. 9(3) (1999), 428-–517.

\bibitem{cht} {\sc V. Chousionis, J. Tyson}, \emph{Marstrand's density theorem in the Heisenberg group}, to appear in Bull. London Math. Soc.

\bibitem{csz} {\sc T. Cie\'slak, M. Szuma\'nska}, \emph{A theorem on measures in dimension $2$ and applications to vortex sheets}, J. Funct. Anal. 266(12) (2014), 6780-–6795.

\bibitem{cm} {\sc T. H. Colding, W. P. Minicozzi II}, \emph{Harmonic Functions on Manifolds}, Annals of Math.(2) 146(3), 725--747 (1997).

\bibitem{GG} {\sc M. Gaczkowski, P. G\'orka}, \emph{Harmonic Functions on Metric Measure Spaces: Convergence and Compactness}, Potential Anal. 31, 203--214 (2009).

\bibitem{gt} {\sc D. Gilbarg, N. Trudinger}, \emph{Elliptic partial differential equations of second order},
Reprint of the 1998 edition, Classics in Mathematics, Springer-Verlag, Berlin, 2001, xiv+517 pp.

\bibitem{gork} {\sc P. G\'orka}, \emph{Campanato theorem on metric measure spaces}, Ann. Acad. Sci. Fenn. Math. 34(2) (2009), 523-–528.

\bibitem{haj} {\sc P. Haj³asz}, \emph{Sobolev spaces on an arbitrary metric space}, Potential Anal. 5(4) (1996), 403–-415.

\bibitem{HajKosk} {\sc P. Haj\l asz, P. Koskela}, \emph{Sobolev Met Poincar\'e},
 Mem. Amer. Math. Soc., 145 (2000).

\bibitem{hm} {\sc P. Haj\l asz, J. Mal\'y}, \emph{On approximate differentiability of the maximal function}, Proc. Amer. Math. Soc. 138(1), 165--174 (2010).

\bibitem{hea} {\sc D. Heath}, \emph{Functions possessing restricted mean value properties}, Proc. Amer. Math. Soc. 41 (1973), 588–-595.

\bibitem{hlnt} {\sc T. Heikkinen, J. Lehrb\"ack. J. Nuutinen, H. Tuominen}, \emph{Fractional maximal functions in metric measure spaces}, Anal. Geom. Metr. Spaces 1 (2013), 147-–162.

\bibitem{hei01}{\sc J. Heinonen}, \emph{ Lectures on analysis on metric spaces}, Universitext, Springer-Verlag, New York, 2001.

\bibitem{he07} {\sc J. Heinonen}, \emph{Nonsmooth calculus}, Bull. Amer. Math. Soc. (N.S.) 44(2) (2007), 163–-232.

\bibitem{hk} {\sc J. Heinonen, P. Koskela}, \emph{Quasiconformal maps in metric spaces with controlled geometry},
        Acta Math. 181, 1--61 (1998).

\bibitem{kpt} {\sc C. Kenig, D. Preiss, T. Toro}, \emph{Boundary structure and size in terms of interior and exterior harmonic measures in higher dimensions}, J. Amer. Math. Soc. 22(3) (2009), 771-–796.

\bibitem{KoMc} {\sc P. Koskela, P. MacManus}, \emph{Quasiconformal mappings and Sobolev spaces}, Studia Math., 131 (1998), 1--17.

\bibitem{kur} {\sc K. Kuratowski}, \emph{Topology}, vol. 2, Academic Press, New York--London, 1968.

\bibitem{llo} {\sc J. Llorente}, \emph{Mean value properties and unique continuation}, Commun. Pure Appl. Anal. 14(1) (2015), 185-–199.

\bibitem{lps} {\sc H. Luiro, M. Parviainen, E. Saksman}, \emph{On the existence and uniqueness of $p$-harmonious functions}, Differential Integral Equations 27(3-4) (2014), 201-–216.


\bibitem{mpr} {\sc  J. Manfredi, M. Parviainen, J. Rossi}, \emph{On the definition and properties of $p$-harmonious functions}, Ann. Sc. Norm. Super. Pisa Cl. Sci. (5) 11(2) (2012), 215-–241.

\bibitem{mar} {\sc J. Marstrand}, \emph{The $(\phi, s)$ regular subsets of $n$-space}, Trans. Amer. Math. Soc. 113 (1964), 369–-392.

\bibitem{masil} {\sc M. Masson, J. Siljander}, \emph{H\"older regularity for parabolic De Giorgi classes in metric measure spaces}, Manuscripta Math. 142(1-2) (2013), 187-–214.

\bibitem{piw} {\sc M. Picardello, W. Woess}, \emph{A converse to the mean value property on homogeneous trees},
Trans. Amer. Math. Soc. 311(1) (1989), 209-–225.

\bibitem{pre} {\sc D. Preiss}, \emph{Geometry of measures in $\R^n$: distribution, rectifiability, and densities}, Ann. of Math. (2) 125(3) (1987), 537-–643.

\bibitem{rans} {\sc A. Ranjan, H. Shah}, \emph{Harmonic manifolds with minimal horospheres},  J. Geom. Anal. 12(4) (2002), 683-–694.

\bibitem{sh} {\sc N. Shanmugalingam}, \emph{Newtonian spaces: an extension of Sobolev spaces to metric measure spaces}, Rev. Mat. Iberoamericana 16(2) (2000), 243-–279.

\bibitem{Sh-harm} {\sc N. Shanmugalingam}, \emph{Harmonic functions on metric spaces}, Illinois J. Math. 45 (2001), 1021--1050.

\bibitem{tol} {\sc X. Tolsa}, \emph{Uniform measures and uniform rectifiability}, J. Lond. Math. Soc. (2) 92(1) (2015), 1-–18.

\bibitem{tod} {\sc L.  Todjihounde}, \emph{Mean-value property on manifolds with minimal horospheres}, J. Aust. Math. Soc. 84(2) (2008),  277-–282.

\bibitem{va}  {\sc J. V\"ais\"al\"a}, \emph{Lectures on $n$-dimensional quasiconformal mappings} Lecture Notes in Mathematics, Vol. 229. Springer--Verlag, Berlin--New York, 1971. xiv+144 pp.

\bibitem{wil} {\sc T. J. Willmore}, \emph{Mean value theorems in harmonic Riemannian spaces}, J. London Math. Soc. 25 (1950), 54-–57.

\bibitem{zuc} {\sc F. Zucca}, \emph{The mean value property for harmonic functions on graphs and trees}, Ann. Mat. Pura Appl. (4) 181(1) (2002), 105-–130.

\end{thebibliography}
\end{document}